\documentclass[reqno,11pt]{amsart}
\usepackage{amsfonts,amsmath,amssymb}
\usepackage{graphicx}

\usepackage{color}

\usepackage[margin=1.35in]{geometry}
\numberwithin{equation}{section}

\def\la{\lambda}

\def\ze{\zeta}

\def\varep{\varepsilon}

\def\al{\alpha}

\def\hat{\widehat}

\def\R{{\mathbb R}}
\def\e{{\varepsilon}}
\def\what{\widehat}

\newtheorem{theorem}{Theorem}[section]

\newtheorem{proposition}[theorem]{Proposition}

\newtheorem{lemma}[theorem]{Lemma}

\newcommand{\RR}{\mathbb{R}}

\newcommand{\sign}{{\rm sgn}\thinspace}

\newcommand{\pa}{\partial}

\allowdisplaybreaks[4]

\begin{document}

\title{Global solutions for the generalized SQG patch equation}

\author{Diego C\'ordoba}
\address{Instituto de Ciencias Matem\'{a}ticas}
\email{dcg@icmat.es}

\author{Javier G\'omez-Serrano} 
\address{Princeton University}
\email{jg27@math.princeton.edu}

\author{Alexandru D. Ionescu}
\address{Princeton University}
\email{aionescu@math.princeton.edu}

\thanks{The first two authors were supported in part by the grant MTM2014-59488-P (Spain) and ICMAT Severo Ochoa projects SEV-2011-008 and SEV-2015-556. The first author was supported in part by a Minerva Distinguished Visitorship at Princeton University. The second author was supported in part by an AMS Simons Travel Grant. Part of this work was done while some of the authors were visiting ICMAT and Princeton University, to which they are grateful for their support. The last author was supported in part by NSF grant DMS-1600028 and by NSF-FRG grant DMS-1463753.}

\begin{abstract}

We consider the inviscid generalized surface quasi-geostrophic equation (gSQG) in a patch setting, where the parameter $\alpha \in (1,2)$. The cases $\alpha = 0$ and $\alpha = 1$ correspond to 2d Euler and SQG respectively, and our choice of the parameter $\alpha$ results in a velocity more singular than in the SQG case. 

Our main result concerns the global stability of the half-plane patch stationary solution, under small and localized perturbations. Our theorem appears to be the first construction of stable global solutions for the gSQG-patch equations. The only other nontrivial global solutions known so far in the patch setting are the so-called V-states, which are uniformly rotating and periodic in time solutions. 


\end{abstract}

\maketitle

%\setcounter{tocdepth}{1}
%\tableofcontents

\section{Introduction}

\subsection{The gSQG equations} In this paper, we consider the generalized surface-quasi\-geo\-stro\-phic equations (gSQG):
\begin{align}\label{Int1}
\left\{ \begin{array}{ll}
\partial_{t}\theta+u\cdot\nabla\theta=0,\quad(x,t)\in\mathbb{R}^2\times\mathbb{R}_+, &\\
u=-\nabla^\perp(-\Delta)^{-1+\frac{\alpha}{2}}\theta,\\
\theta_{|t=0}=\theta_0,
\end{array} \right.
\end{align}
where $\alpha \in (0,2)$. The case $\alpha = 1$ corresponds to the surface quasi-geostrophic (SQG) equation and the limiting case $\alpha = 0$ refers to the 2D incompressible Euler equation. The case $\alpha = 2$ produces stationary solutions.

These are so-called active scalar equations, which have been originally introduced and studied in the setting of sufficiently smooth solutions $\theta$. The equations \eqref{Int1}  have also been analyzed extensively in the natural setting of the so-called $\alpha$-\textit{patches}, which are solutions for which $\theta$ is  a step function
\begin{align}\label{pa1}
  \theta(x,t) =
  \left\{
 \begin{array}{ll}
   \theta_1, \text{ if } \ \ x \in \Omega(t) \\ 
   \theta_2, \text{ if }  \ \ x \in \Omega(t) ^c. \\
   \end{array}
  \right.
\end{align}
Here $\Omega(0)\subset\mathbb{R}^2$ is a regular set given by the initial distribution of $\theta$, $\theta_1$ and $\theta_2$ are constants, and $\Omega(t)$ is the evolution of $\Omega(0)$ under the induced velocity field. 

In this case, the evolution of a patch can be determined by looking just at the evolution of its boundary, thus reducing the problem to a nonlocal one-dimensional equation for the boundary of $\Omega(t)$. More precisely, the evolution equation for the interface of an $\alpha$-patch, which we parametrize as $z:I\to\mathbb{R}^2$, $z(x) = (z_1(x),z_2(x))$, can be written as
 \begin{align}
\label{Int2}
 \partial_t z(x,t) = -(\theta_2 - \theta_1)C(\al) \int_{I} \frac{ \partial_x z (x,t) - \partial_x z(x-y,t) }{ \vert z(x,t) - z(x-y, t) \vert^{\alpha}} dy + c(x,t)z_x(x,t).
 \end{align}
Here $I\subseteq\mathbb{R}$ is an interval (usually $I=[0,2\pi]$ in the case of bounded patches or $I=\mathbb{R}$ for unbounded patches), the presence of the function $c$ has to do with the flexibility in parametrizing the curve, and the normalizing constant $C(\al)$ is given by
\begin{align*}
 C(\al) = \frac{1}{2\pi} \frac{\Gamma(\al/2)}{2^{1-\al}\Gamma((2-\al)/2))}.
\end{align*}

\subsubsection{Local regularity} The local regularity theory for the equations \eqref{Int1} and \eqref{Int2} is generally well understood, starting with the work of Constantin--Majda--Tabak \cite{Constantin-Majda-Tabak:formation-fronts-qg} and Held--Pierrehumbert--Garner--Swanson \cite{Held-Pierrehumbert-Garner-Swanson:sqg-dynamics}. As expected, data with sufficient smoothness lead to local in time unique solutions that propagate the regularity of the initial data (see for example Rodrigo \cite{Rodrigo:evolution-sharp-fronts-qg}, Gancedo \cite{Gancedo:existence-alpha-patch-sobolev}, and Chae--Constantin--Cordoba-Gancedo--Wu \cite{Chae-Constantin-Cordoba-Gancedo-Wu:gsqg-singular-velocities} for some regularity results of this type).

Weak solutions have also been constructed, starting with the work of Resnick \cite{Resnick:phd-thesis-sqg-chicago} on global weak solutions in $L^2$ in the SQG case $\al = 1$. See also \cite{Marchand:existence-regularity-weak-solutions-sqg}, \cite{Chae-Constantin-Cordoba-Gancedo-Wu:gsqg-singular-velocities}, and \cite{Nahmod-Pavlovic-Staffilani-Totz:global-invariant-measures-gsqg} for more general classes of weak solutions. Very recently, Buckmaster--Shkoller--Vicol \cite{Buckmaster-Shkoller-Vicol:nonuniqueness-sqg} proved lack of uniqueness of weak solutions for the SQG equation in certain spaces less regular than $L^{2}$.

\subsubsection{Dynamical formation of singularities} The problem of whether the SQG evolution can lead to finite time singularities is a challenging open problem both in the smooth case \eqref{Int1} and in the patch case \eqref{Int2}.

In the smooth case \eqref{Int1}, the early numerical simulations in \cite{Constantin-Majda-Tabak:formation-fronts-qg} indicated a possible singularity in the form of a hyperbolic saddle closing in finite time. However, C\'ordoba \cite{Cordoba:nonexistence-hyperbolic-blowup-qg} showed that such a scenario cannot actually lead to singularities, and bounded the growth by a quadruple exponential (see also \cite{Cordoba-Fefferman:growth-solutions-qg-2d-euler} and \cite{Deng-Hou-Li-Yu:non-blowup-2d-sqg}). The same scenario was recently revisited with bigger computational power and improved algorithms by Constantin--Lai--Sharma--Tseng--Wu \cite{Constantin-Lai-Sharma-Tseng-Wu:new-numerics-sqg}, yielding no evidence of blowup and the depletion of the hyperbolic saddle past the previously computed times.  More recently, Scott \cite{Scott:scenario-singularity-quasigeostrophic}, starting from elliptical configurations, proposed a candidate that appears to develop filamentation and, after a few cascades, blowup of $\nabla \theta$.

In the patch case \eqref{Int2}, numerical simulations looking for singularities of the interface have been performed by different authors. There are at least two scenarios that suggest a possible formation of singularities. The first one (computed by C\'ordoba--Fontelos--Mancho--Rodrigo \cite{Cordoba-Fontelos-Mancho-Rodrigo:evidence-singularities-contour-dynamics}), involves the evolution of two patches, and the simulation suggests an asymptotically self-similar singular scenario in which the distance between the two patches goes to zero in finite time while simultaneously the curvature of the boundaries blows up. Scott--Dritschel \cite{Scott-Dritschel:self-similar-sqg} started from an elliptical patch with a large ratio between its axes and found numerically that it may develop a self-similar singularity with a blowup of the curvature in the case $\al = 1$. This is consistent with the rule out of splash singularities by Gancedo--Strain \cite{Gancedo-Strain:absence-splash-muskat-SQG}.

These recent simulations appear to suggest (convincingly) the possibility of dynamical  formation of singularities. However, we emphasize that no rigorous results are known. The best result so far regarding fast growth is due to Kiselev--Nazarov \cite{Kiselev-Nazarov:simple-energy-pump-sqg}, who constructed solutions that started arbitrarily small but grew arbitrarily large in finite time. More recently, Kiselev--Ryzhik--Yao--Zlatos \cite{Kiselev-Ryzhik-Yao-Zlatos:singularity-alpha-patch-boundary} introduced a new gSQG-patch model, with a fixed boundary, and proved the formation of finite time singularities in this model for certain patches that touch the boundary at all times. At this point it is unclear whether such a scenario can lead to singularities in the classical gSQG models considered here.

\subsubsection{Global regularity and rotating solutions} The construction of nontrivial global solutions for the gSQG equations is also a challenging problem, both in the smooth and in the patch case. In fact, the only non-stationary global solutions that are known in the patch setting are very special rotating solutions. These solutions, which are periodic in time and evolve by rotating with constant angular velocity around their center of mass, are known as V-states. 

Deem--Zabusky \cite{Deem-Zabusky:vortex-waves-stationary} were the first to discover the V-states numerically in the patch case. Other authors have later improved the methods and numerically computed larger classes (see for example \cite{Wu-Overman-Zabusky:steady-state-Euler-2d,Elcrat-Fornberg-Miller:stability-vortices-cylinder,LuzzattoFegiz-Williamson:efficient-numerical-method-steady-uniform-vortices,Saffman-Szeto:equilibrium-shapes-equal-uniform-vortices}).

Hassainia--Hmidi \cite{Hassainia-Hmidi:v-states-generalized-sqg} have rigorously proved the existence of V-states in the case $0 < \alpha < 1$. They were able to show the existence of convex V-states with $C^{k}$ boundary regularity. In \cite{Castro-Cordoba-GomezSerrano:existence-regularity-vstates-gsqg}, Castro--C\'ordoba--G\'omez-Serrano were able to prove existence and $C^{\infty}$ regularity of convex global rotating solutions for the remaining open cases: $\al \in [1,2)$ for the existence, $\al \in (0,2)$ for the regularity. This boundary regularity was subsequently improved to analytic in \cite{Castro-Cordoba-GomezSerrano:analytic-vstates-ellipses}.

The problem of constructing rotating periodic in time solutions is more challenging in the smooth case \eqref{Int1}. Such solutions have only been constructed very recently by Castro--C\'ordoba--G\'omez-Serrano \cite{Castro-Cordoba-GomezSerrano:global-smooth-solutions-sqg} who found a smooth 3-fold solution that rotates uniformly (both in time and space) by perturbing from a smooth annular profile. See also \cite{Castro-Cordoba-GomezSerrano:uniformly-rotating-smooth-euler}. We remark that Dritschel \cite{Dritschel:exact-rotating-solution-sqg} had constructed nontrivial global rotating solutions with $C^{1/2}$ regularity. 

\subsection{The main theorem} Our goal in this paper is to initiate the study of {\it{stable}} global solutions of the equations \eqref{Int1} and \eqref{Int2}. Such stable solutions cannot be periodic in time and their construction requires a different mechanism. 

A natural way to look for families of global stable solutions is to perturb around certain explicit stationary solutions of the equation. This approach has been successful to produce nontrivial global solutions for many difficult quasilinear evolutions, such as the Einstein-vacuum equations, plasma models, or water-wave models. In the case of time reversible equations the main mechanism that sometimes leads to global solutions is the mechanism of {\it{dispersion}}. 

In our case of the gSQG equations, one could start by perturbing around the trivial solution $\theta\equiv 0$ of the equation \eqref{Int1}. However, there is no source of dispersion in this case and it is not clear to us how to control the solution beyond the natural time of existence $T_\e\approx \e^{-1}$ corresponding to data of size $\e$.  

One could also start from the observation that all radial functions are stationary solutions of the gSQG equations, and look for global solutions that start as small perturbations of radial functions. A natural such problem would be to consider the gSQG-patch equation \eqref{Int2}, and start with data that is a small perturbation of the characteristic function of a disk. Numerical simulations in this case seem to suggest the existence of long-term (perhaps global) smooth solutions for the gSQG-patch equation \eqref{Int2}, starting from certain small perturbations of a characteristic function of a ball of radius $1$. So far, however, we have not been able to analyze this scenario rigorously.  

In this paper we consider a simpler scenario, namely we perturb around the half-plane stationary solution corresponding to the straight interface
\begin{equation*}
z_1(x)=x,\qquad z_2(x)=0.
\end{equation*}  
For simplicity, we will assume that $C(\al)(\theta_1 - \theta_2) = 1$, $c(x,t) = 0$ and $z_1(x,t) = x$. This choice yields the following equation for $z_2(x,t) \equiv h(x,t)$:
\begin{equation}\label{zxc2}
\partial_th(x,t) = \int_{\mathbb{R}} \frac{h_{x}(x,t) - h_{x}(x-y,t)}{\big(|h(x,t)-h(x-y,t)|^2+y^2\big)^{\alpha/2}}dy.
\end{equation}
We will consider solutions that decay at $\pm \infty$, so the integral in \eqref{zxc2} is well defined for $\al \in (1,2)$.

At the linear level, the dynamics of solutions of \eqref{zxc2} are determined by the equation 
\begin{align}\label{linearization}
\pa_{t} \hat{h}(\xi,t) = i \Lambda(\xi)\hat{h}(\xi,t),\qquad \Lambda(\xi):=\gamma |\xi|^{\al-1}\xi,
\end{align}
where $\hat{h}(\xi,t)$ is the Fourier transform of $h(x,t)$ and $\gamma\in(0,\infty)$ is a constant. We notice that this linearized equation has dispersive character, due to the dispersion relation $\Lambda$, which is related to the stationary solution we perturb around. Thus one can hope to prove global regularity and decay. This is precisely our main theorem:

\begin{theorem}\label{MainThm}
Assume $\alpha\in(1,2)$, and let $N_0:=20$ and $N_1:=4$. Then there is a constant $\overline{\e}=\overline{\e}(\al)$ such that for all initial-data $h_0:\mathbb{R}\to \mathbb{R}$ satisfying the smallness conditions
\begin{equation}\label{th1}
\|h_0\|_{H^{N_0\al}}+\|x\partial_xh_0\|_{H^{N_1\al}}\leq\e_0\leq\overline{\e}
\end{equation}
there is a unique global solution $h\in C([0,\infty):H^{N_0\al}(\mathbb{R}))$ of the evolution equation \eqref{zxc2} with $h(0)=h_0$. Moreover, the solution $h$ satisfies the slow growth energy bounds
\begin{equation}\label{th2}
\|h(t)\|_{H^{N_0\al}}+\|Sh(t)\|_{H^{N_1\al}}\lesssim \e_0(1+t)^{p_0},\qquad t\in[0,\infty),
\end{equation}
where $S:=\al t\partial_t+x\partial_x$ is the scaling vector-field associated to the linear equation \eqref{linearization} and $p_0:=10^{-7}(2-\al)$, and the sharp pointwise decay bounds
\begin{align}\label{th3}
\big(2^{k/2} + 2^{N_2\al k} \big) {\|P_k h(t)\|}_{L^\infty} \lesssim \e_0(1+t)^{-1/2},\qquad t\in[0,\infty),\,k\in\mathbb{Z},
\end{align}
where $P_k$ denote the standard Littlewood-Paley projections and $N_2:=8$.  
\end{theorem}

Our proof provides more information about the global solution $h$ as part of the bootstrap argument. In fact, the solution satisfies the main bounds \eqref{zxc6} in Proposition \ref{MainProp}. At a qualitative level, the solution $h$ remains uniformly bounded in a suitable $Z$-norm and  undergoes nonlinear (modified) scattering as $t\to\infty$. See the discussion in subsection \ref{MainIdeas} below. 

\subsection{Main ideas of the proof}\label{MainIdeas} The equation \eqref{zxc2} is a time reversible quasilinear equation. The classical mechanism to prove global regularity in such a situation has two main steps:

\setlength{\leftmargini}{1.8em}
\begin{itemize}
  \item[(1)] Prove energy estimates to propagate control of high order Sobolev and weighted norms;
\smallskip
  \item[(2)] Prove dispersion and decay of the solution over time.
\end{itemize}

The interplay of these two aspects has been present since the seminal work of Klainerman \cite{Klainerman:vector-fields-wave-equation,Klainerman:null-structure-global-existence-wave-equation} on nonlinear wave equations and vector-fields,
Shatah \cite{Shatah:normal-forms-quadratic-klein-gordon} and Simon \cite{Si} on $3$d Klein-Gordon equations and normal forms, Christodoulou-Klainerman \cite{CK} on the stability of the Minkowski space-time,
and Delort \cite{DelortKGE} on $1$d Klein-Gordon equations.

In the last few years new methods have emerged in the study of global solutions of quasilinear evolutions, inspired by the advances in semilinear theory.
The basic idea is to combine the classical energy and vector-fields methods with refined analysis of the Duhamel formula, using the Fourier transform and carefully constructed ``designer'' norms. This is the main idea of the ``method of space-time resonances'' of Germain-Masmoudi-Shatah \cite{GeMaSh,Germain-Masmoudi-Shatah:global-solutions-gravity-water-waves-annals} and Gustafson-Nakanishi-Tsai \cite{GNT1},
and of the work on plasma models and water wave models of the last author and his collaborators, in \cite{IP1,GIP,DIP,Ionescu-Pusateri:global-solutions-water-waves-2d,Ionescu-Pusateri:model2d,Ionescu-Pusateri:global-solutions-water-waves-2d-surface-tension,Deng-Ionescu-Pausader-Pusateri:global-solutions-gravity-capillary-water-waves-3d}.

We describe now in some detail these two main aspects of our proof.

\subsubsection{Energy estimates} We would like to control the growth in time of two energy-type quantities: the high order Sobolev norms of our solutions and their weighted norms. More precisely, we would like to prove that the solution $h$ satisfies energy bounds with slow growth of the form
\begin{equation}\label{explain1}
\|h(t)\|_{H^{N_0\al}}+\|Sh(t)\|_{H^{N_1\al}}\lesssim \e_0(1+t)^{p_0},\qquad t\in[0,\infty),
\end{equation}
where $S:=\al t\partial_t+x\partial_x$ is the scaling vector-field associated to the linearized equation, and $p_0\ll 1$. For this we use a paradifferential reduction, similar to the idea used recently in the study of water-wave models in \cite{Alazard-Burq-Zuily:water-wave-surface-tension, ABZ2, AD1, Ionescu-Pusateri:global-solutions-water-waves-2d-surface-tension, Deng-Ionescu-Pausader-Pusateri:global-solutions-gravity-capillary-water-waves-3d}.\footnote{Alternatively, one could try to use a change of variables as in \cite{Gancedo:existence-alpha-patch-sobolev}. This works well to control the Sobolev norms, but seems to lead to problems in the analysis of the weighted norms involving the vector-field $S$. Because of this we prefer to use here the more robust paradifferential approach.} We examine the equation \eqref{zxc2} and start by rewriting it using paradifferential calculus in the form
\begin{equation}\label{explain2}
\partial_t h=i\Lambda h+iT_{\Sigma}h+E.
\end{equation} 
Here $\Lambda(\xi)=\gamma |\xi|^{\al-1}\xi$ is as in \eqref{linearization}, $\Sigma$ is an explicit symbol of order $\alpha$, which depends quadratically on $h$, $T$ is the paradifferential operator in Weyl quantization (see subsection \ref{ParaDiffCalc} for definitions and simple properties), and $E$ is a suitable error term that satisfies cubic bounds in $h$ and does not lose derivatives (relative to $h$).

The formula \eqref{explain2} gives a good idea about the structure of the nonlinearity, but is not completely adequate to prove energy estimates. This is because the symbol $\Sigma$ in not real-valued, thus the operator $T_\Sigma$ is not self-adjoint. In fact, $\Sigma$ can be written in the form 
\begin{equation*}
\Sigma=\Sigma^\al+\Sigma^1+\Sigma^{\al-1},
\end{equation*} 
where $\Sigma^\al$ and $\Sigma^1$ are real-valued symbols of order $\al$ and $1$, and $\Sigma^{\al-1}$ is a purely imaginary symbol of order $\al-1$. To prove energy estimates we need one more step. Precisely, we define the renormalized variable $h^\ast$ (the so-called ``good variable'') by $h^\ast:=T_bh$, for a suitable symbol $b$ of order $0$. This symbol is constructed in such a way that the good variable $h^\ast$ satisfies a better evolution equation of the form
\begin{equation}\label{explain3}
\partial_t h^\ast=i\Lambda h^\ast+iT_{\Sigma^\al+\Sigma^1}h^\ast+E'.
\end{equation}
The resulting error term $E'$ still satisfies good cubic bounds with no derivative loss. 

The equation \eqref{explain3} is now suitable to prove energy bounds, first for the good variable $h^\ast$, and then for the original variable $h$. The only additional ingredient that is needed to prove the bounds \eqref{explain1} is sharp pointwise decay of the solution, i.e an estimate of the form
\begin{align}\label{explain5}
\big(2^{k/2} + 2^{N_2\al k} \big) {\|P_k h(t)\|}_{L^\infty} \lesssim \e_1(1+t)^{-1/2},\qquad t\in[0,\infty),\,k\in\mathbb{Z}.
\end{align}
This follows from the main bootstrap assumption and linear estimates.           

\subsubsection{Dispersion and decay} To close the bootstrap argument we need to prove dispersion, in a sufficiently precise way so as to be able to recover the sharp pointwise decay bounds \eqref{explain5}. Since all the energy estimates have a small $(1+t)^{p_0}$ loss, this requires an independent argument, which does not rely directly on these energy estimates. 

We use the $Z$-norm method. More precisely, we define a suitable norm, called the $Z$-norm, in such a way that $\|h(t)\|_{Z}$ is uniformly bounded as $t\to\infty$,
\begin{equation}\label{explain7}
\|h(t)\|_{Z}\lesssim \e_0.
\end{equation} 
The precise choice of the $Z$-norm is important, since control of the $Z$-norm has to complement suitably the energy control proved in the first step. Here we use a type of norms introduced recently in 2D water-wave models by Ionescu--Pusateri \cite{Ionescu-Pusateri:global-solutions-water-waves-2d,Ionescu-Pusateri:model2d,Ionescu-Pusateri:global-solutions-water-waves-2d-surface-tension}, 
\begin{equation}\label{explain8}
{\| f \|}_Z := {\big\| \big(|\xi|^{1/2+\beta} + |\xi|^{N_2\alpha+1} \big) \widehat{f}(\xi) \big\|}_{L^\infty_\xi},\qquad N_2=8,\,\beta=(2-\al)/10.
\end{equation}

To prove \eqref{explain7} we start by defining the linear profile of the solution $v(t)=e^{-it\Lambda}h(t)$. Then we write the Duhamel formula in terms of the profile $v$. The main contribution comes from the cubic nonlinear term
\begin{equation*}
i\int_{\R\times\R}m_1(\eta_1,\eta_2,\xi-\eta_1-\eta_2)e^{it\Phi(\xi,\eta_1,\eta_2)}\widehat{v}(\eta_1,t)\widehat{v}(\eta_2,t)\widehat{v}(\xi-\eta_1-\eta_2,t)\,d\eta_1 d\eta_2,
\end{equation*}
where $\Phi(\xi,\eta_1,\eta_2)=-\Lambda(\xi)+\Lambda(\eta_1)+\Lambda(\eta_2)+\Lambda(\xi-\eta_1-\eta_2)$ and $m_1$ is a suitable multiplier. 

This term is not integrable in time, due to the contribution of the space-time re\-so\-nan\-ces. To eliminate these contributions we need to add a nonlinear correction to the profile $v$. More precisely, we define the (modified) nonlinear profile $v^\ast$ by the formulas
\begin{equation*}
\widehat{v^\ast}(\xi,t):=\widehat{v}(\xi,t)e^{iL(\xi,t)},\qquad L(\xi,t):=\widetilde{c}(\xi)\int_0^t|\widehat{v}(\xi,s)|^2\frac{1}{s+1}\,ds,
\end{equation*}  
where $\widetilde{c}$ is a suitable function (see \eqref{nf50} for the precise formulas). Then we show that the nonlinear profile $v^\ast(t)$ converges in the $Z$-norm as $t\to\infty$, at a suitable rate, and prove the uniform bounds \eqref{explain7}.

\subsection{Organization} The rest of the paper is concerned with the proof of Theorem \ref{MainThm}. In section \ref{Prelims} we introduce the main notation, define the $Z$-norm, prove some important lemmas, and state the main bootstrap Proposition \ref{MainProp}. 

In the remaining two sections we prove Proposition \ref{MainProp}, along the lines described above. In section \ref{Ene} we prove the energy estimates, using paradifferential calculus, while in section \ref{Poi1} we prove the dispersive estimates, using the Duhamel formula and Fourier analysis.

\section{Preliminaries and the main bootstrap proposition}\label{Prelims}

\subsection{Notation and basic lemmas}\label{notation}
In this subsection we summarize some of our main notation and recall several basic formulas and estimates.
We fix an even smooth function $\varphi: \R\to[0,1]$ supported in $[-8/5,8/5]$ and equal to $1$ in $[-5/4,5/4]$,
and define
\begin{equation*}
\varphi_k(x) := \varphi(x/2^k) - \varphi(x/2^{k-1}) , \qquad \varphi_{\leq k}(x):=\varphi(x/2^k),
  \qquad\varphi_{\geq k}(x) := 1-\varphi(x/2^{k-1}),
\end{equation*}
for any $k\in\mathbb{Z}$. Let $P_k$, $P_{\leq k}$, and $P_{\geq k}$ the operators defined by the Fourier multipliers $\varphi_k$,
$\varphi_{\leq k}$, and $\varphi_{\geq k}$ respectively. For any interval $I\subseteq\mathbb{R}$ let
\begin{equation*}
P_{I}:=\sum_{k\in\mathbb{Z}\cap I}P_k.
\end{equation*} 

For any $k\in\mathbb{Z}$ let
\begin{equation*}
k^+:=\max(k,0),\qquad k^-:=\min(k,0).
\end{equation*}

\subsubsection{Multipliers and associated operators} We will often work with multipliers $m:\mathbb{R}^2\to\mathbb{C}$ or $m:\mathbb{R}^3\to\mathbb{C}$, and operators defined by such multipliers. We define the class of symbols
\begin{equation}
\label{Sinfty}
S^\infty := \{m: \R^d \to \mathbb{C} : \,m \text{ continuous and } {\| m \|}_{S^\infty} := {\|\mathcal{F}^{-1}(m)\|}_{L^1} < \infty \}.
\end{equation}
We summarize below some properties of multipliers and associated operators (see \cite[Lemma 5.2]{Ionescu-Pusateri:global-solutions-water-waves-2d} for the proof).

\begin{lemma}\label{touse}
 (i) We have $S^\infty\hookrightarrow L^\infty(\mathbb{R}^d)$. If $m,m'\in S^\infty$ then $m\cdot m'\in S^\infty$ and
\begin{equation}\label{al8}
\|m\cdot m'\|_{S^\infty}\leq \|m\|_{S^\infty}\|m'\|_{S^\infty}.
\end{equation}
Moreover, if $m\in S^\infty$, $A:\mathbb{R}^d\to\mathbb{R}^d$ is an invertible linear transformation, $v\in\mathbb{R}^d$, and $m_{A,v}(\xi):=m(A\xi+v)$ then
\begin{equation}\label{al8.1}
\|m_{A,v}\|_{S^\infty}=\|m\|_{S^\infty}.
\end{equation}

(ii) Assume $p,q,r\in[1,\infty]$ satisfy $1/p+1/q=1/r$, and $m\in S^\infty$. Then, for any $f,g\in L^2(\mathbb{R})$,
\begin{equation}\label{mk6}
\begin{split}
&\|T_m(f,g)\|_{L^r} \lesssim \|m\|_{S^\infty} \|f\|_{L^p} \|g\|_{L^q}\\
&\text{where }\qquad\mathcal{F}\{T_m(f,g)\}(\xi):=\int_{\mathbb{R}}m(\xi-\eta,\eta)\widehat{f}(\xi-\eta)\widehat{g}(\eta)\,d\eta.
\end{split}
\end{equation}
In particular, if $1/p+1/q+1/r=1$ then
\begin{equation}\label{mk6.1}
\Big|\int_{\mathbb{R}\times\mathbb{R}}m(\xi-\eta,\eta)\widehat{f}(\xi-\eta)\widehat{g}(\eta)\widehat{h}(\xi)\,d\eta d\xi\Big|\lesssim \|m\|_{S^\infty} \|f\|_{L^p} \|g\|_{L^q}\|h\|_{L^r}.
\end{equation}

(iii) More generally, if $d\geq 2$, $p_1,\ldots p_d,q\in[1,\infty]$ satisfy $1/p_1+\ldots+1/p_d=1/q$, and $f_1,\ldots,f_d\in L^2(\mathbb{R})$ then
\begin{equation}\label{mk6.5}
\|T_m(f_1,\ldots,f_d)\|_{L^q}\lesssim \|m\|_{S^\infty}\|f_1\|_{L^{p_1}}\ldots \|f_d\|_{L^{p_d}},
\end{equation}
where
\begin{equation*}
\begin{split}
\mathcal{F}\{T_m(f_1,\ldots, f_d)\}(\xi):=&\int_{\mathbb{R}^{d-1}}m(\xi-\eta_2-\ldots-\eta_d,\eta_2,\ldots,\eta_{d})\\
&\times\widehat{f_1}(\xi-\eta_2-\ldots-\eta_{d})\widehat{f_2}(\eta_2)\ldots \widehat{f_d}(\eta_{d})\,d\underline{\eta},
\end{split}
\end{equation*}

and $\underline{\eta} = (\eta_2,\ldots,\eta_{d})$.
\end{lemma}

Moreover, if $f_1,\ldots,f_d$ are suitable functions defined on $\mathbb{R}\times I$ and $S=\al t\partial_t+x\partial_x$ is the scaling vector-field then
\begin{equation}\label{mk6.9}
S[T_m(f_1,\ldots,f_d)]=T_m(Sf_1,f_2,\ldots,f_d)+\ldots+T_m(f_1,\ldots,f_{d-1},Sf_d)+T_{\widetilde{m}}(f_1,\ldots,f_d),
\end{equation}
where
\begin{equation}\label{mk6.10}
\widetilde{m}(\xi_1,\ldots \xi_d):=-(\xi_1\partial_{\xi_1}+\ldots+\xi_d\partial_{\xi_d})m(\xi_1,\ldots \xi_d).
\end{equation}

\subsubsection{An interpolation lemma} We will use the following simple lemma, see \cite[Lemma 4.3]{Ionescu-Pusateri:model2d} for the proof.

\begin{lemma}\label{interpolation}
For any $k \in \mathbb{Z}$, and $f\in L^2(\mathbb{R})$ we have
\begin{equation}
\label{interp1}
{\big\| \widehat{P_kf} \big\|}_{L^\infty}^2 \lesssim {\big\| P_kf \big\|}_{L^1}^2 \lesssim
  2^{-k} {\|\widehat{f}\|}_{L^2} \big[ 2^k {\|\partial \widehat{f}\|}_{L^2} + {\|\widehat{f}\|}_{L^2} \big].
\end{equation}
\end{lemma}

\subsubsection{A dispersive estimate} The following lemma is our main linear dispersive estimate:

\begin{lemma}\label{dispersive}
Assume that $\Lambda(\xi)=\gamma\xi|\xi|^{\alpha-1}$ as before. Then, for any $t\in\mathbb{R}\setminus\{0\}$, $k\in\mathbb{Z}$, and $f\in L^2(\mathbb{R})$ we have
\begin{equation}\label{disperse}
 \|e^{i t \Lambda}P_kf\|_{L^\infty}\lesssim |t|^{-1/2}2^{k(1-\alpha/2)}\|\widehat{f}\|_{L^\infty}+|t|^{-3/4}2^{-k(3\alpha/4-1/2)}\big[2^k\|\partial \widehat{f}\|_{L^2}+\|\widehat{f}\|_{L^2}\big]
\end{equation}
and
\begin{equation}\label{disperseEa}
 \|e^{i t \Lambda}P_kf\|_{L^\infty}\lesssim |t|^{-1/2}2^{k(1-\alpha/2)}\|f\|_{L^1}.
\end{equation}
\end{lemma}

\begin{proof} This is similar to the proof of Lemma 4.2 in \cite{Ionescu-Pusateri:model2d}. For \eqref{disperse} it suffices to prove that
\begin{equation}\label{disp1}
\begin{split}
\Big|\int_{\mathbb{R}}&e^{it\Lambda(\xi)}e^{ix\xi}\widehat{f}(\xi)\varphi_k(\xi)\,d\xi\Big|\\
&\lesssim |t|^{-1/2}2^{k(1-\alpha/2)}\|\widehat{f}\|_{L^\infty}+|t|^{-3/4}2^{-k(3\alpha/4-1/2)}\big[2^k\|\partial \widehat{f}\|_{L^2}+\|\widehat{f}\|_{L^2}\big]
\end{split}
\end{equation}
for any $t\in\mathbb{R}\setminus\{0\}$ and $x\in\mathbb{R}$. The left-hand side of \eqref{disp1} is clearly bounded by $C2^k\|\widehat{f}\|_{L^\infty}$. Therefore in proving \eqref{disp1} we may assume $|t|\geq 2^{20}2^{-\alpha k}$.

Let $\Psi(\xi):=t\Lambda(\xi)+x\xi$ and notice that 
\begin{equation}\label{disp2.02}
\Psi'(\xi)=t\gamma \alpha|\xi|^{\alpha-1}+x\qquad\text{ and }\qquad \Psi''(\xi)=t\gamma \alpha(\alpha-1)|\xi|^{\alpha-2}\mathrm{sgn}\,(\xi).
\end{equation}
If $|\Psi'(\xi)|\gtrsim |t|2^{(\alpha-1)k}$ in the support of the integral in the left-hand side of \eqref{disp1} then integrate by parts in $\xi$ to estimate this integral by
\begin{equation*}
C\int_{\mathbb{R}}\frac{1}{|t|2^{(\alpha-1)k}}\big[|\partial\widehat{f}(\xi)|+2^{-k}|\widehat{f}(\xi)|\big]\varphi_{[k-2,k+2]}(\xi)\,d\xi\lesssim \frac{2^{-k/2}}{|t|2^{(\alpha-1)k}}\big[2^k\|\partial \widehat{f}\|_{L^2}+\|\widehat{f}\|_{L^2}\big],
\end{equation*}
which suffices to prove \eqref{disp1}, in view of the assumption $|t|\geq 2^{-\alpha k}$.

It remains to prove the bound \eqref{disp1} when $|t\gamma \alpha |\xi|^{\alpha-1}+x|\ll |t|2^{(\alpha-1)k}$ for some $\xi$ with $|\xi|\in[2^{k-4},2^{k+4}]$. This is possible only if $x/t<0$. Let $\xi_0^{\pm}\in\mathbb{R}$ denote the solutions of the equation $\Psi'(\xi)=0$, i.e.
\begin{equation*}
 \xi_0^{\pm}:=\pm \Big|\frac{-x}{\gamma\alpha t}\Big|^{1/(\alpha-1)},
\end{equation*}
and notice that $|\xi_0^{\pm}|\approx 2^{k}$. We estimate 
\begin{equation}\label{disp6}
\Big|\int_{\mathbb{R}}e^{it\Lambda(\xi)}e^{ix\xi}\widehat{f}(\xi)\varphi_k(\xi)\,d\xi\Big|\leq\sum_{l\leq k+40}\big[|J_l^{+}|+|J_l^{-}|\big],
\end{equation}
where, for any $l\geq l_0$,
\begin{equation*}
J_{l}^{\pm}:=\int_{\mathbb{R}}e^{i\Psi(\xi)}\cdot \widehat{f}(\xi)\varphi_k(\xi)\mathbf{1}_{\pm}(\xi)\varphi_l(\xi-\xi_0^{\pm})\,d\xi.
\end{equation*}

Clearly
\begin{equation}\label{disp6.1}
 \sum_{2^l\leq 2^{k(1-\alpha/2)}|t|^{-1/2}}|J_l^{\pm}|\lesssim \sum_{2^l\leq 2^{k(1-\alpha/2)}|t|^{-1/2}}2^{l}\|\widehat{f}\|_{L^\infty}\lesssim 2^{k(1-\alpha/2)}|t|^{-1/2}\|\widehat{f}\|_{L^\infty}.
\end{equation}
On the other hand, since $|\Psi'(\xi)|\gtrsim 2^{k(\alpha-2)}|t|2^l$ in the support of the integral defining $J_l^{\pm}$, if $2^l\in [2^{k(\alpha-1/2)}|t|^{-1/2},2^{k+40}]$ then we can integrate by parts to estimate
\begin{equation*}
|J_l^{\pm}|\lesssim \frac{2^{k(2-\alpha)}}{|t|2^l}\int_{\mathbb{R}}\big[|\partial\widehat{f}(\xi)|+2^{-l}|\widehat{f}(\xi)|\big]\varphi_{\leq l+4}(\xi-\xi_0^{\pm})\,d\xi\lesssim \frac{2^{k(2-\alpha)}}{|t|2^l}\|\widehat{f}\|_{L^\infty}+\frac{2^{k(2-\alpha)}}{|t|2^{l/2}}\|\partial\widehat{f}\|_{L^2}.
\end{equation*}
This suffices to control the sum of $|J_l^{\pm}|$ over $2^l\geq 2^{k(1-\alpha/2)}|t|^{-1/2}$ as claimed in \eqref{disp1}. The full bound \eqref{disp1} follows using also \eqref{disp6} and \eqref{disp6.1}.

To prove the dispersive bound \eqref{disperseEa} we notice that $e^{it\Lambda}P_kf=f\ast G_k$ where 
\begin{equation*}
G_k=c\int_{\mathbb{R}}e^{it\Lambda(\xi)}e^{ix\xi}\varphi_k(\xi)\,d\xi.
\end{equation*}
Clearly, $\|G_k\|_{L^\infty}\lesssim 2^k$. It follows from \eqref{disp1} that $\|G_k\|_{L^\infty}\lesssim |t|^{-1/2}2^{k(1-\alpha/2)}+|t|^{-3/4}2^{k(1-3\alpha/4)}$. The dispersive bound \eqref{disperseEa} follows by considering the two cases $2^{-\alpha k}\leq |t|$ and $2^{-\alpha k}\geq |t|$. 
\end{proof}

\subsection{Linearization and expansion of the nonlinearity} Recall the main equation \eqref{zxc2}. To linearize it we write
\begin{equation}\label{Poi4}
\partial_t h=i\Lambda h+\mathcal{N},
\end{equation}  
where
\begin{equation}\label{Poi4.5}
i(\Lambda h)(x):=\int_{\mathbb{R}} \frac{h'(x) - h'(x-y)}{|y|^\alpha}dy,
\end{equation}
and
\begin{equation}\label{Poi5}
\mathcal{N}(x):=\int_{\mathbb{R}} \Big\{\frac{h_{x}(x) - h_{x}(x-y)}{\big(|h(x)-h(x-y)|^2+y^2\big)^{\alpha/2}}-\frac{h_{x}(x) - h_{x}(x-y)}{|y|^\alpha}\Big\}\,dy.
\end{equation}
We examine first the linear part, and notice that
\begin{equation*}
i\widehat{\Lambda h}(\xi)=i\xi\widehat{h}(\xi)\int_{\mathbb{R}} \frac{1-e^{-iy\xi}}{|y|^\alpha}dy.
\end{equation*}
An easy calculation shows that $\widehat{\Lambda h}(\xi)=\Lambda(\xi)\widehat{h}(\xi)$ where
\begin{equation}\label{Poi7}
\Lambda(\xi):=\gamma\xi|\xi|^{\alpha-1},\qquad \gamma:=\int_{\mathbb{R}} \frac{1-\cos y}{|y|^\alpha}dy=\frac{2\Gamma(2-\al)\sin(\al\pi/2)}{\al-1}\in(0,\infty).
\end{equation} 

Using the formal expansion formula
\begin{equation}\label{Nonl1}
(1+\rho)^{-\alpha/2}=1+\sum_{n\geq 1}d_n\rho^n,\qquad d_n=\frac{(-\alpha/2)(-\alpha/2-1)\cdot\ldots\cdot (-\alpha/2-n+1)}{n!},
\end{equation}
where $|\rho|<1$, we write 
\begin{equation}\label{Nonl2}
\mathcal{N}=\sum_{n\geq 1}\mathcal{N}_{n},
\end{equation}
where
\begin{equation}\label{Nonl3}
\mathcal{N}_{n}(x):=d_{n}\int_{\mathbb{R}} \frac{h_{x}(x) - h_{x}(x-y)}{|y|^\alpha}\Big(\frac{h(x)-h(x-y)}{y}\Big)^{2n}\,dy.
\end{equation}

Since   
\begin{equation*}
\frac{h(x)-h(x-y)}{y}=\frac{1}{2\pi}\int_{\mathbb{R}}\widehat{h}(\eta)e^{ix\eta}\frac{1-e^{-i y\eta}}{y}\,d\eta
\end{equation*}
we can symmetrize and rewrite $\mathcal{N}_n$ in the Fourier space in the form
\begin{equation}\label{nonl4}
\widehat{\mathcal{N}_n}(\xi) = \frac{i}{2n+1}\int_{\RR^{2n}} \hat{h}(\eta_1) \hat{h}(\eta_2)\ldots \hat{h}(\eta_{2n}) \hat{h}\left(\xi-\sum_{i=1}^{2n}\eta_{i}\right) m_n\left(\eta_1,\eta_2,\ldots,\eta_{2n},\xi-\sum_{i=1}^{2n}\eta_{i}\right) d\underline{\eta},
\end{equation}
where $\underline{\eta}=(\eta_1,\ldots,\eta_{2n})$ and 
\begin{align}\label{non4.1}
m_n(\lambda_1,\lambda_2,\ldots,\lambda_{2n+1}):= \frac{d_n}{(2\pi)^{2n}}\left(\sum_{i=1}^{2n+1}\lambda_i\right)\int_{\RR} \prod_{i=1}^{2n+1}\left(\frac{1-e^{-i\lambda_i y}}{y}\right) \frac{dy}{|y|^{\al-1}\sign(y)}.
\end{align}
We will show later, in Lemma \ref{mbounds}, that the multipliers $m_n$, $n\geq 1$, are real-valued, odd, and homogeneous of degree $2n+\al$. They also satisfy suitable symbol-type bounds in $S^\infty$ which allow us to estimate the associated multilinear operators.

\subsection{Weyl paradifferential calculus}\label{ParaDiffCalc} We will use paradifferential calculus in section \ref{Ene} to prove high order energy estimates. In this subsection we summarize the results we need, which are all standard. We refer the reader to \cite[Appendix A]{Deng-Ionescu-Pausader-Pusateri:global-solutions-gravity-capillary-water-waves-3d} for the (elementary) proofs.

We recall first the definition of paradifferential operators (Weyl quantization): given a symbol $a=a(x,\zeta):\mathbb{R}\times\mathbb{R}\to\mathbb{C}$, we define the operator $T_a$ by
\begin{equation}\label{Tsigmaf2}
\begin{split}
\mathcal{F}\left\{T_{a}f\right\}(\xi)=\frac{1}{2\pi}\int_{\mathbb{R}}\chi\Big(\frac{\vert\xi-\eta\vert}{\vert\xi+\eta\vert}\Big)\widetilde{a}(\xi-\eta,(\xi+\eta)/2)\widehat{f}(\eta)d\eta,
\end{split}
\end{equation}
where $\widetilde{a}$ denotes the partial Fourier transform of $a$ in the first coordinate and $\chi=\varphi_{\le-20}$. We define the Poisson bracket between two symbols $a$ and $b$ by the formula
\begin{equation}\label{Poisson}
\{a,b\}:=\partial_xa\partial_\zeta b-\partial_\zeta a\partial_x b.
\end{equation}

For $q\in [1,\infty]$ and $l\in \mathbb{R}$ we define $\mathcal{L}^q_l$ as the space of symbols defined by the norm
\begin{equation}\label{nor1}
\begin{split}
&\|a\|_{\mathcal{L}^q_l}:=\sup_{\zeta\in\mathbb{R}^2}(1+|\zeta|^2)^{-l/2}\|\,|a|(.,\zeta)\|_{L^q_x},\\
&|a|(x,\zeta):=\sum_{\beta\leq 20,\,\alpha\leq 4}(1+|\zeta|^2)^{\beta/2}|(\partial^\beta_\zeta \partial^\alpha_x a)(x,\zeta)|.
\end{split}
\end{equation}
The index $l$ is called the {\it{order}} of the symbol, and it measures the contribution of the symbol in terms of derivatives on $f$. Notice that we have the simple product rule
\begin{equation}\label{nor2}
\|ab\,\|_{\mathcal{L}^p_{l_1+l_2}}\lesssim \|a\|_{\mathcal{L}^q_{l_1}}\|b\|_{\mathcal{L}^r_{l_2}},\qquad 1/p=1/q+1/r.
\end{equation}

An important property of paradifferential operators is that they behave well with respect to products. More precisely:

\begin{lemma}\label{PropProd} (i) If $1/p=1/q+1/r$ and $k\in\mathbb{Z}$, and $l\in [-10,10]$ then
\begin{equation}
\label{LqBdTa}
\Vert P_kT_af\Vert_{L^p}\lesssim 2^{lk^+}\Vert a\Vert_{\mathcal{L}^q_l}\Vert P_{[k-2,k+2]}f\Vert_{L^r}.
\end{equation}

(ii) Assume $p,q_1,q_2,r\in[1,\infty]$ satisfy $1/p=1/q_1+1/q_2+1/r$, $k\in\mathbb{Z}$, and $a\in \mathcal{L}^{q_1}_{l_1}$, $b\in\mathcal{L}^{q_2}_{l_2}$ are symbols, $l_1,l_2\in [-4,4]$. Then 
\begin{equation}
\label{WeylB1}
2^{k^+}\Vert P_k(T_aT_b-T_{ab})f\Vert_{L^p}\lesssim (2^{l_1k^+}\Vert a\Vert_{\mathcal{L}^{q_1}_{l_1}})(2^{l_2k^+}\Vert b\Vert_{\mathcal{L}^{q_2}_{l_2}})\cdot\Vert P_{[k-4,k+4]}f\Vert_{L^r},
\end{equation}
and
\begin{equation}
\label{WeylB2}
2^{2k^+}\Vert P_k(T_aT_b-T_{ab}-(i/2)T_{\{a,b\}})f\Vert_{L^p}\lesssim (2^{l_1k^+}\Vert a\Vert_{\mathcal{L}^{q_1}_{l_1}})(2^{l_2k^+}\Vert b\Vert_{\mathcal{L}^{q_2}_{l_2}})\cdot\Vert P_{[k-4,k+4]}f\Vert_{L^r}.
\end{equation}
\end{lemma} 

The point of these bounds is the gain of one derivative in \eqref{WeylB1} and the gain of 2 derivatives in the more precise formula \eqref{WeylB2}.

With $S=\alpha t\partial_t+x\partial_x$ (the scaling vector-field), notice that we have the identities
\begin{equation}\label{Alu2.1}
\begin{split}
&\widehat{Sg}(\xi,t)=(\al t\partial_t-\xi\partial_\xi-I)\widehat{g}(\xi,t),\qquad (\xi\partial_{\xi}+\eta\partial_{\eta})\Big[\chi\Big(\frac{\vert\xi-\eta\vert}{\vert\xi+\eta\vert}\Big)\Big]\equiv 0.
\end{split}
\end{equation}
We record below some simple properties of paradifferential operators.

\begin{lemma}\label{PropSym} 

(i) If $a\in \mathcal{L}_0^\infty$ is real-valued then $T_a$ is a bounded self-adjoint operator on $L^2$. Moreover,
\begin{equation}\label{Alu2}
\overline{T_af}=T_{a'}\overline{f},\quad\text{ where }\quad a'(y,\zeta):=\overline{a(y,-\zeta)}.
\end{equation} 

(ii) For suitable functions $f$ defined on $\mathbb{R}\times[0,T]$ and symbols $a$ defined on $\mathbb{R}\times\mathbb{R}\times[0,T]$ we have
\begin{equation}\label{Alu3}
S (T_af)=T_a(Sf)+T_{S_{x,\zeta}a}f\quad \text{ where }\quad (S_{x,\zeta}a)(x,\zeta,t)=(S_{x,t}a)(x,\zeta,t)-(\zeta\partial_{\zeta}a)(x,\zeta,t).
\end{equation} 
\end{lemma}

These properties follow easily from definitions; for \eqref{Alu3} one uses also the identities \eqref{Alu2.1} and integration by parts in $\eta$.

\subsection{The $Z$-norm and the main bootstrap proposition} For any function $f \in L^2(\R)$ let
\begin{equation}
\label{Znorm}
{\| f \|}_Z := {\big\| \big(|\xi|^{1/2+\beta} + |\xi|^{N_2\alpha+1} \big) \widehat{f}(\xi) \big\|}_{L^\infty_\xi} ,
\end{equation}
where $N_2:=8$ and $\beta:=(2-\alpha)/10$. Our main Theorem \ref{MainThm} follows, by local existence theory and a continuity argument, from the following main proposition:

\begin{proposition}[Main bootstrap]\label{MainProp}
Assume that $\alpha\in(1,2)$,
\begin{equation}\label{zxc1}
\begin{split}
& N_0:=20,\qquad N_1:=4,\qquad N_2:=8,\qquad \beta:=(2-\alpha)/10,\\
& p_0:= 10^{-6}\beta,\qquad 0<\varep_0\leq\varep_1\leq\varep_0^{2/3}\ll \beta.
\end{split}
\end{equation}
Assume $T\geq 1$ and $h\in C\big([0,T]:H^{N_0\alpha}\big)$ is a real-valued solution of the system \eqref{zxc2}. Let $v$ denote the linear profile of the solution,
\begin{equation}\label{prof}
v(t) := e^{-it\Lambda}h(t).
\end{equation}
Recall that $S=\alpha t\partial_t+x\partial_x$. Assume that, for any $t\in[0,T]$,
\begin{equation}\label{zxc4}
\langle t\rangle^{-p_0}\big[\|h(t)\|_{H^{N_0 \al }}+\|Sh(t)\|_{H^{N_1 \al}}+\|x\partial_xv(t)\|_{L^2}\big]+\|v(t)\|_{Z}\leq \e_1,
\end{equation}
where $\langle t\rangle=1+t$. Assume also that the initial data $h_0=h(0)$ satisfy the stronger bounds
\begin{equation}\label{zxc5}
\|h_0\|_{H^{N_0 \al}}+\|x\partial_xh_0\|_{H^{N_1 \al}}\leq\e_0.
\end{equation}
Then the solution satisfies the improved bounds, for any $t\in[0,T]$,
\begin{equation}\label{zxc6}
\langle t\rangle^{-p_0}\big[\|h(t)\|_{H^{N_0 \al }}+\|Sh(t)\|_{H^{N_1 \al}}+\|x\partial_xv(t)\|_{L^2}\big]+\|v(t)\|_{Z}\lesssim \e_0.
\end{equation}
\end{proposition}

We notice that the assumption \eqref{zxc5} implies the desired bounds \eqref{zxc6} at time $t=0$, using also Lemma \ref{interpolation}. The rest of the paper is concerned with proving the bounds \eqref{zxc5} at all times $t\in[0,T]$. This is done in two steps, in Propositions \ref{proEEZ} and \ref{proZ}.

\section{Energy estimates}\label{Ene}

In this section we prove the following proposition.

\begin{proposition}\label{proEEZ}
With the assumptions in Proposition \ref{MainProp}, we have, for any $t\in[0,T]$,
\begin{equation}\label{Ens2}
\|h(t)\|_{H^{N_0 \al}}+\|Sh(t)\|_{H^{N_1 \al}}+\|x\partial_xv(t)\|_{L^2}\lesssim \e_0\langle t\rangle^{p_0}.
\end{equation}
\end{proposition}

The rest of the section is concerned with the proof of this proposition. We start with two lemmas, concerning pointwise decay of the solution and bounds on the multipliers $m_n$. The main step is to rewrite the nonlinearity in paradifferential form, in Lemma \ref{MainhDec}. Then we construct a suitable renormalization of the variable $h$ and prove the desired energy bounds.

\subsection{Two lemmas} We start by proving sharp pointwise decay bounds on $h(t)$. 

\begin{lemma}\label{Poi2}
For any $t\in[0,T]$ and $k\in\mathbb{Z}$ we have
\begin{align}\label{prodecay0}
\big(2^{k/2-\beta k} + 2^{N_2\al k} \big) {\|P_k h(t)\|}_{L^\infty} \lesssim \e_1\langle t \rangle^{-1/2}.
\end{align}
\end{lemma}

\begin{proof}
We may assume $t\geq 1$ and let $P'_k:=P_{[k-2,k+2]}$. We use Lemma \ref{dispersive} with $f=P_kv(t) = e^{-it\Lambda}P_kh(t)$ to show that
\begin{align}
\label{prodecay2}
{\| P_k h(t) \|}_{L^\infty} \lesssim t^{-1/2}2^{k(1-\alpha/2)} {\|\widehat{P_k^\prime v}(t)\|}_{L^\infty}
   + t^{-3/4} 2^{-k(3\alpha/4-1/2)} \big[ 2^k {\|\partial \widehat{P_k^\prime v}(t)\|}_{L^2} + {\|\widehat{P_k^\prime v}(t)\|}_{L^2} \big]
\end{align}
and
\begin{equation}
\label{prodecay3}
{\| P_k h(t) \|}_{L^\infty} \lesssim t^{-1/2} 2^{k(1-\alpha/2)} {\| P_k^\prime v(t)\|}_{L^1}.
\end{equation}
It follows from \eqref{zxc4} that 
\begin{equation}\label{prodecay4}
\begin{split}
2^k\|\partial \widehat{P'_k v}(t)\|_{L^2} + \|\widehat{P'_k v}(t)\|_{H^{N_0\al}}&\lesssim \e_1 t^{p_0},\\
(2^{k/2+\beta k}+2^{(N_2\alpha+1)k})\|\widehat{P'_k v}(t)\|_{L^\infty}&\lesssim \e_1.
\end{split}
\end{equation}
Therefore
\begin{align*}
{\| P_k h(t) \|}_{L^\infty} \lesssim t^{-1/2}2^{k(1-\alpha/2)} \frac{\e_1}{2^{k/2+\beta k} + 2^{(N_2\alpha+1)k}}
   + t^{-3/4} 2^{-k(3\alpha/4-1/2)} \e_1 t^{p_0}.
\end{align*}
Since $1-\alpha/2=5\beta$, see \eqref{zxc1}, this is enough to show that
\begin{align}
\label{prodecay20}
\sup_{k\in \mathbb{Z}, \, \langle t \rangle^{-1/4} \leq 2^k \leq \langle t \rangle^{100p_0}} \big(2^{k/2-\beta k} + 2^{\alpha N_2k} \big) {\|P_k h(t)\|}_{L^\infty}
\lesssim \e_1 \langle t \rangle^{-1/2}.
\end{align}

Combining \eqref{prodecay3}, with Lemma \ref{interpolation}, and \eqref{prodecay4} we see that
\begin{equation*}
{\| P_k h(t) \|}_{L^\infty} \lesssim t^{-1/2} 2^{k(1-\alpha/2)} \e_1 \langle t\rangle^{p_0} 2^{-k/2}2^{-\al N_0k^+/2}.
\end{equation*}
It follows that
\begin{align}
\label{prodecay30}
\sup_{k\in \mathbb{Z}, \, 2^k \leq \langle t \rangle^{-1/4}} 2^{k/2-\beta k} {\|P_k h(t)\|}_{L^\infty}
  + \sup_{k\in \mathbb{Z}, \, 2^k \geq \langle t \rangle^{100p_0}} 2^{\al N_2k} {\|P_k h(t)\|}_{L^\infty}
  \lesssim \e_1 \langle t\rangle^{-1/2} .
\end{align}
The desired bounds \eqref{prodecay0} follow from \eqref{prodecay20} and \eqref{prodecay30}.
\end{proof}

We record now some properties of several multipliers that are important in the analysis. 

\begin{lemma}\label{mbounds}

(i) The multipliers $m_n$ defined in \eqref{non4.1} are real-valued, odd, and homogeneous of degree $2n+\al$, 
\begin{equation}\label{non4.2}
\begin{split}
&m_n(-\la_1,\ldots,-\la_{2n+1})=-m_n(\la_1,\ldots,\la_{2n+1}),\\
&m_n(\rho\la_1,\ldots,\rho\la_{2n+1})=\rho^{2n+\al}m_n(\la_1,\ldots,\la_{2n+1}),\qquad\rho>0.
\end{split}
\end{equation}

(ii) For any $k_1,\ldots,k_{2n+1}\in\mathbb{Z}$ and $n\geq 1$ let
\begin{equation}\label{nonl6}
m_n^{k_1,\ldots,k_{2n+1}}(\mu_1,\ldots,\mu_{2n+1}):=m_n(\mu_1,\ldots,\mu_{2n+1})\cdot \varphi_{k_1}(\mu_1)\ldots\varphi_{k_{2n+1}}(\mu_{2n+1}).
\end{equation}
Let $k_{\mathrm{max}}:=\max(k_1,\ldots,k_{2n+1})$ and $k_{\mathrm{min}}:=\min(k_1,\ldots,k_{2n+1})$. Then
\begin{equation}\label{nonl7}
\|m_n^{k_1,\ldots,k_{2n+1}}\|_{S^\infty}\leq C^n2^{k_1+\ldots+k_{2n+1}}2^{(\al-1)k_{\mathrm{max}}}
\end{equation}
for some constant $C=C_\al\geq 1$. Moreover, if $n=1$ and $j\in\{1,2,3\}$ then
\begin{equation}\label{nonl7.5}
\|\partial_{\mu_j}m_1^{k_1,k_2,k_3}\|_{S^\infty}\lesssim 2^{-k_j}2^{k_{\mathrm{min}}}2^{(1+\alpha)k_{\mathrm{max}}}.
\end{equation}

(iii) Let
\begin{equation}\label{Sig5}
p_n(\la_1,\ldots,\la_{2n}):=\int_{[0,1]^{2n}}\Big|\sum_{i=1}^{2n}s_i\lambda_i\Big|^{\al-1}\,ds_1\ldots d_{s_{2n}}.
\end{equation}
The multipliers $p_n$ are real-valued and homogeneous of degree $\al-1$, and satisfy the bounds
\begin{equation}\label{nonl13}
\big\|\mathcal{F}^{-1}\big\{p_n(\la_1,\ldots,\la_{2n})\cdot \varphi_{k_1}(\la_1)\ldots\varphi_{k_{2n}}(\la_{2n})\big\}\big\|_{L^1}\leq C^n 2^{(\al-1)\max(k_1,\ldots,k_{2n})}.
\end{equation}
\end{lemma}

\begin{proof} (i) This is a consequence of the definitions (alternatively, one can use the formula \eqref{nonl5} derived below). 

(ii) We write, using the formula \eqref{non4.1}, 
\begin{equation}\label{nonl8}
\begin{split}
\mathcal{F}^{-1}&(m_n^{k_1,\ldots,k_{2n+1}})(x_1,\ldots,x_{2n+1})=A_n2^{k_1+\ldots+k_{2n+1}}\sum_{i=1}^{2n+1}2^{k_i}\int_{\mathbb{R}}\frac{dy}{|y|^{\alpha-1}\sign(y)}\\
&\times\frac{\psi'_0(2^{k_i}x_i)-\psi'_0(2^{k_i}(x_i-y))}{y}\prod_{j\in[1,2n+1],\,j\neq i}\frac{\psi_0(2^{k_j}x_j)-\psi_0(2^{k_j}(x_j-y))}{y},
\end{split}
\end{equation}
where $A_n\leq C^n$ is a constant, $\psi_0$ is the inverse Fourier transform of $\varphi_0$, and $\psi'_0$ is its derivative. We notice that if $|y|\approx 2^p$ and $\widetilde{\psi}\in\{\psi_0,\psi'_0\}$ then
\begin{equation}\label{nonl9}
\int_{\mathbb{R}}|\widetilde{\psi}(2^{k}x)-\widetilde{\psi}(2^{k}(x-y))|\,dx\lesssim \min(2^{-k},2^p).
\end{equation}
Therefore, using also that $\al\in(1,2)$,
\begin{equation*}
\begin{split}
\big\|\mathcal{F}^{-1}(m_n^{k_1,\ldots,k_{2n+1}})\big\|_{L^1}&\lesssim_n 2^{k_1+\ldots+k_{2n+1}}2^{k_{\mathrm{max}}}\sum_{p\in\mathbb{Z}}2^{p(2-\alpha)}\min(1,2^{-k_1-p})\ldots\min(1,2^{-k_{2n+1}-p})\\
&\lesssim_n 2^{k_1+\ldots+k_{2n+1}}2^{(\al-1)k_{\mathrm{max}}},
\end{split}
\end{equation*}
where $X\lesssim_n Y$ (here and in the rest of the paper) means that there is a constant $C_\al\geq 1$ such that $X\leq C_\al^nY$. This gives the desired bounds \eqref{nonl7}.

To prove \eqref{nonl7.5} we notice that $\mu_j$ differentiation corresponds to multiplication by $x_j$. Then we notice that if $|y|\approx 2^p$ then
\begin{equation}\label{nonl11}
\int_{\mathbb{R}}|x||\widetilde{\psi}(2^{k}x)-\widetilde{\psi}(2^{k}(x-y))|\,dx\lesssim \min(2^{-k},2^p)(2^{-k}+2^p)\lesssim 2^{-k}2^p.
\end{equation} 
Therefore, estimating as before,
\begin{equation*}
\begin{split}
\big\|&\mathcal{F}^{-1}(\partial_{\mu_j}m_1^{k_1,k_2,k_3})\big\|_{L^1}\\
&\lesssim 2^{k_1+k_2+k_3}2^{k_{\mathrm{max}}}\sum_{p\in\mathbb{Z}}(2^{-k_j}+2^p)2^{p(2-\alpha)}\min(1,2^{-k_1-p})\min(1,2^{-k_2-p})\min(1,2^{-k_3-p})\\
&\lesssim 2^{-k_j}2^{k_{\mathrm{min}}}2^{(1+\alpha)k_{\mathrm{max}}}.
\end{split}
\end{equation*}

(iii) We start from the formula
\begin{equation*}
\gamma|w|^{\al-1}=\int_{\mathbb{R}}\frac{1-e^{-iwy}}{|y|^\al}\,dy,
\end{equation*}
which follows from \eqref{Poi7}. Therefore
\begin{equation*}
\begin{split}
\mathcal{F}^{-1}\big\{p_n(\la_1,\ldots,&\la_{2n})\cdot \varphi_{k_1}(\la_1)\ldots\varphi_{k_{2n}}(\la_{2n})\big\}(x_1,\ldots,x_{2n})\\
&=B_n\int_{[0,1]^{2n}\times\mathbb{R}}\frac{\prod_{i=1}^{2n}[2^{k_i}\psi_0(2^{k_i}x_i)]-\prod_{i=1}^{2n}[2^{k_i}\psi_0(2^{k_i}(x_i-ys_i))]}{|y|^\al}\,dy d\underline{s},
\end{split}
\end{equation*}
where $\underline{s}=(s_1,\ldots,s_{2n})$. Using \eqref{nonl9} it follows that
\begin{equation*}
\begin{split}
\big\|\mathcal{F}^{-1}\big\{p_n(\la_1,\ldots,\la_{2n})\cdot \varphi_{k_1}&(\la_1)\ldots\varphi_{k_{2n}}(\la_{2n})\big\}\big\|_{L^1}\lesssim_n\int_{[0,1]^{2n}\times\mathbb{R}}\frac{\sum_{i=1}^{2n}\min(1,2^{k_i}ys_i)}{|y|^\al}\,dy d\underline{s}\\
&\lesssim_n\int_{\mathbb{R}}\frac{\min(1,2^{\max(k_1,\ldots,k_{2n})}y)}{|y|^\al}\,dy\lesssim_n2^{(\al-1)\max(k_1,\ldots,k_{2n})},
\end{split}
\end{equation*}
as desired.
\end{proof}

\subsection{Paradifferential formulation} The main issue when proving energy estimates is to avoid the potential loss of derivative. We do this here using a paradifferential reduction. Our next lemma is the main step in the proof of Proposition \ref{proEEZ}.

\begin{lemma}\label{MainhDec}
We have
\begin{equation}\label{yui1}
\partial_t h=i\Lambda h+iT_{\Sigma}h+E,
\end{equation}
where

(i) The error term $E=E(h)\in C([0,T]:H^{N_0\al})$ satisfies the bounds
\begin{equation}\label{plo1}
\|E(t)\|_{H^{N_0\alpha}}+\|SE(t)\|_{H^{N_1\alpha}}\lesssim\e_1^2\langle t\rangle^{-1+p_0};
\end{equation}

(ii) The symbol $\Sigma=\Sigma(h)$ decomposes as
\begin{equation}\label{plo2}
\Sigma=\Sigma^\alpha+\Sigma^1+\Sigma^{\al-1},
\end{equation}
where $\Sigma^\al$, $\Sigma^1$, and $\Sigma^{\al-1}$ are symbols of order $\al$, $1$, and $\al-1$ given by
\begin{equation}\label{Sig1}
\begin{split}
\widetilde{\Sigma^\al}(\rho,\zeta):=\Lambda(\zeta)&\varphi_{\geq 30}(\zeta)\sum_{n\geq 1} \frac{d_n(-1)^n}{(2\pi)^{2n-1}}\int_{\mathbb{R}^{2n-1}}\Big(\prod_{i=1}^{2n-1}\widetilde{g_n}(\eta_i,\ze)\Big)\widetilde{g_n}\Big(\rho-\sum_{i=1}^{2n-1}\eta_i,\ze\Big)\,d\underline{\eta},
\end{split}
\end{equation}
\begin{equation}\label{Sig2}
\begin{split}
\widetilde{\Sigma^1}(\rho,\zeta):=-\gamma\zeta&\varphi_{\geq 30}(\zeta)\sum_{n\geq 1} \frac{d_n(-1)^n}{(2\pi)^{2n-1}}\int_{\mathbb{R}^{2n-1}}p_n\Big(\underline{\eta},\rho-\sum_{i=1}^{2n-1}\eta_i\Big)\\
&\times\Big(\prod_{i=1}^{2n-1}\widetilde{g_n}(\eta_i,\ze)\Big)\widetilde{g_n}\Big(\rho-\sum_{i=1}^{2n-1}\eta_i,\ze\Big)\,d\underline{\eta},
\end{split}
\end{equation}
\begin{equation}\label{Sig3}
\widetilde{\Sigma^{\al-1}}(\rho,\zeta):=\frac{\rho}{2\ze}\cdot\widetilde{\Sigma^\al}(\rho,\zeta),
\end{equation}
where $\underline{\eta}:=(\eta_1,\ldots,\eta_{2n-1})$, $p_n$ are defined as in \eqref{Sig5}, and
\begin{equation}\label{Sig4}
\widetilde{g_n}(\rho,\ze):=\rho\widehat{h}(\rho)\varphi_{\leq -4n-30}(\rho/\zeta).
\end{equation}

The symbols $\Sigma^\al$ and $\Sigma^1$ are real-valued and, see also definitions \eqref{nor1} and \eqref{Alu3}, 
\begin{equation}\label{Sig8}
\begin{split}
&\|\Sigma^\kappa\|_{\mathcal{L}^\infty_\kappa}\lesssim \e_1^2\langle t\rangle^{-1},\qquad\kappa\in\{\al,1,\al-1\},\\
&\|\Sigma^\kappa\|_{\mathcal{L}^2_\kappa}+\|S_{x,\zeta}\Sigma^\kappa\|_{\mathcal{L}^2_\kappa}\lesssim \e_1^2\langle t\rangle^{-1/2+p_0},\qquad\kappa\in\{\al,1,\al-1\}.
\end{split}
\end{equation} 
\end{lemma}

\begin{proof} We use the formulas \eqref{nonl4}. The idea is to extract "error" terms that can be estimated like in \eqref{plo1}, and identify the main contributions to the symbols in $\Sigma$. Most of the bounds we prove rely on the symbol bounds \eqref{nonl7} and \eqref{nonl13}, Lemma \ref{touse}, and the bounds on $h$,
\begin{equation}\label{nok2.7}
\|h(t)\|_{H^{N_0\alpha}}+\|Sh(t)\|_{H^{N_1\alpha}}\lesssim\e_1\langle t\rangle^{p_0},\qquad \sup_{k\in\mathbb{Z}}\,(2^{k/2-\beta k}+2^{N_2\al k})\|P_kh(t)\|_{L^\infty}\lesssim \e_1\langle t\rangle^{-1/2}.
\end{equation}
These bounds follow from the bootstrap assumptions \eqref{zxc4} and Lemma \ref{Poi2}.

{\bf{Step 1.}} We bound first the contribution when the two highest frequencies are proportional. For $n\geq 1$ we define the multilinear operators $L_n$ by
\begin{equation}\label{cru0}
\begin{split}
\mathcal{F}\{L_n[f_1,\ldots,f_{2n},f_{2n+1}]\}(\xi)&:=\frac{i}{2n+1}\int_{\RR^{2n}} \widehat{f_1}(\eta_1)\ldots \widehat{f_{2n}}(\eta_{2n})\\
&\times \widehat{f_{2n+1}}\Big(\xi-\sum_{i=1}^{2n}\eta_{i}\Big)m_n\Big(\eta_1,\eta_2,\ldots,\eta_{2n},\xi-\sum_{i=1}^{2n}\eta_{i}\Big) d\underline{\eta}.
\end{split}
\end{equation}
Notice that $\mathcal{N}_n=L_n[h,\ldots,h,h]$, see \eqref{nonl4}. For $\underline{k}:=(k_1,\ldots,k_{2n+1})\in\mathbb{Z}^{2n+1}$ let
\begin{equation}\label{cru10}
\mathcal{N}_{n;\underline{k}}:=L_n[P_{k_1}h,\ldots,P_{k_{2n+1}}h].
\end{equation}

Let
\begin{equation}\label{cru11}
X_n:=\{\underline{k}=(k_1,\ldots,k_{2n+1})\in\mathbb{Z}^{2n+1}:k_{\mathrm{max}}\leq \max(0,k_{\mathrm{sec}})+4n+40\},
\end{equation}
where $k_{\mathrm{sec}}$ is the second larger of the numbers $k_1,\ldots,k_{2n+1}$. We will show that
\begin{equation}\label{cru13}
\sum_{\underline{k}\in X_n}\big\{\|\mathcal{N}_{n;\underline{k}}(t)\|_{H^{N_0\alpha}}+\|S\mathcal{N}_{n;\underline{k}}(t)\|_{H^{N_1\alpha}}\big\}\lesssim_n\e_1^{2n+1}\langle t\rangle^{-1+p_0}.
\end{equation}

We use \eqref{mk6.5}, \eqref{nonl7}, and \eqref{nok2.7}. Assuming, for example, that $k_{2n+1}=k_{\mathrm{max}}$, we estimate
\begin{equation*}
\begin{split}
\|\mathcal{N}_{n;\underline{k}}(t)\|_{H^{N_0\al}}&\lesssim_n 2^{N_0\al\max(k_{2n+1},0)}2^{k_1+\ldots+k_{2n}}2^{\al k_{2n+1}}\|P_{k_{2n+1}}h(t)\|_{L^2}\prod_{i=1}^{2n}\|P_{k_i}h(t)\|_{L^\infty}\\
&\lesssim_n \prod_{i=1}^{2n}\frac{\varep_1\langle t\rangle^{-1/2}2^{k_i}}{2^{k_i/2}+2^{N_2\al k_i}}2^{\al \max(k_{2n+1},0)}\varep_1\langle t\rangle^{p_0}.
\end{split}
\end{equation*} 
 Thus
\begin{equation*}
\sum_{\underline{k}\in X_n}\|\mathcal{N}_{n;\underline{k}}(t)\|_{H^{N_0\alpha}}\lesssim_n\e_1^{2n+1}\langle t\rangle^{-1+p_0}.
\end{equation*}

To estimate $S\mathcal{N}_{n;\underline{k}}$ we use also the identities \eqref{mk6.9}--\eqref{mk6.10}. In our case, since $m_n$ is homogeneous, see \eqref{non4.2}, we have
\begin{equation}\label{cru15}
(\xi_1\partial_{\xi_1}+\ldots+\xi_{2n+1}\partial_{\xi_{2n+1}})m_n(\xi_1,\ldots,\xi_{2n+1})=(2n+\al)m_n(\xi_1,\ldots,\xi_{2n+1}).
\end{equation}
Estimating as before, with the term containing the vector-field in $L^2$ and the others in $L^\infty$, and noticing that $N_1\leq N_2-4$ we have
\begin{equation*}
\begin{split}
\|S\mathcal{N}_{n;\underline{k}}\|_{H^{N_1\al}}\lesssim_n \e_1^{2n+1}\langle t\rangle^{p_0-1}2^{\min(k_{\mathrm{min}},0)/2}2^{-\max(k_{\mathrm{max}},0)}
\end{split}
\end{equation*}
for any $\underline{k}\in X_n$. The desired conclusion \eqref{cru13} follows.

It remains to control the contribution of $\underline{k}\notin X_n$. Notice that, by symmetry,
\begin{equation}\label{cru20}
\sum_{\underline{k}\notin X_n}\mathcal{N}_{n;\underline{k}}=(2n+1)\sum_{k\geq 4n+41}L_n[P_{\leq k-4n-41}h,\ldots,P_{\leq k-4n-41}h,P_kh].
\end{equation} 

{\bf{Step 2.}} To deal with the contribution of \eqref{cru20} we derive first more favorable formulas for the multipliers $m_n$. Using the mean value theorem
\begin{align*}
\prod_{i=1}^{2n+1}&\Big(\frac{1-e^{-i\lambda_i y}}{y}\Big)\frac{1}{|y|^{\al-1}\sign(y)}=\frac{1}{|y|^{\al-1}\sign(y)}\int_{[0,1]^{2n+1}}\prod_{i=1}^{2n+1}i\lambda_{i}e^{-is_i\lambda_i y} d\underline{s} \nonumber \\
& = i(-1)^n\Big(\prod_{i=1}^{2n+1}\lambda_{i}\Big)\int_{[0,1]^{2n+1}} e^{-i\sum_{i=1}^{2n+1}s_i\lambda_iy} \frac{1}{|y|^{\al-1}\sign(y)}d\underline{s}.
\end{align*}
We recall the following identity for $0 < g < 1$,
\begin{align*}
 \int_{\RR} \frac{e^{iwy}}{|y|^{g}\sign(y)} dy = 2i |w|^{g-1}\sign(w)\Gamma(1-g)\cos\left(\frac{g\pi}{2}\right).
\end{align*}
We use these identities with $g = \al-1$ and $w = -\sum_{i=1}^{2k+1}s_i\lambda_i$, together with \eqref{non4.1}. Thus
\begin{equation}\label{nonl5}
\begin{split}
m_n(\lambda_1,&\lambda_2,\ldots,\lambda_{2n+1})=\Gamma(2-\al)\sin\left(\frac{\al\pi}{2}\right) \frac{2d_n(-1)^n}{(2\pi)^{2n}}\\
&\times\Big(\prod_{i=1}^{2n+1}\lambda_{i}\Big)\Big(\sum_{i=1}^{2n+1}\lambda_i\Big)\int_{[0,1]^{2n+1}}\Big|\sum_{i=1}^{2n+1}s_i\lambda_i\Big|^{\al-2}\sign\Big(\sum_{i=1}^{2n+1}s_i\lambda_i\Big)d\underline{s},
\end{split}
\end{equation}
where $\underline{s}=(s_1,\ldots,s_{2n+1})$. 

Assume that
\begin{equation}\label{cru1}
\lambda_{2n+1}\geq 2^{2n+10}\lambda_{j}\text{ for any }j\in\{1,\ldots,2n\}.
\end{equation}
Let $\mu_i = \frac{\lambda_{i}}{\lambda_{2n+1}}$. The formula \eqref{nonl5} gives
\begin{equation*}
\begin{split}
m_n(\lambda_1,\lambda_2,\ldots,\lambda_{2n+1})&=C_{n,\al}(\al-1)|\lambda_{2n+1}|^{\al-1}\Big(\prod_{i=1}^{2n}\lambda_{i}\Big)\Big(\sum_{i=1}^{2n+1}\lambda_i\Big)\\
&\times\int_{[0,1]^{2n+1}}\Big|s_{2n+1} + \sum_{i=1}^{2n}s_i\mu_i\Big|^{\al-2}\sign\Big(s_{2n+1}+\sum_{i=1}^{2n}s_i\mu_i\Big)d\underline{s},
\end{split}
\end{equation*}
where, with $\gamma$ as in \eqref{Poi7},
\begin{align*}
C_{n,\al}:=\frac{\Gamma(2-\al)}{\al-1}\sin\left(\frac{\al\pi}{2}\right) \frac{2d_n(-1)^n}{(2\pi)^{2n}}=\gamma\frac{d_n(-1)^n}{(2\pi)^{2n}}.
\end{align*}
Integrating in $s_{2n+1}$ we obtain
\begin{equation}\label{cru7}
\begin{split}
m_n(\lambda_1,\ldots,\lambda_{2n+1}) & = C_{n,\al}|\lambda_{2n+1}|^{\al-1}\Big(\sum_{i=1}^{2n+1}\lambda_i\Big) \Big(\prod_{i=1}^{2n}\lambda_{i}\Big)\\
&\times\int_{[0,1]^{2n}}\Big(\Big|1 + \sum_{i=1}^{2n}s_i\mu_i\Big|^{\al-1}-\Big|\sum_{i=1}^{2n}s_i\mu_i\Big|^{\al-1}\Big)ds_{1}\ldots ds_{2n}.
\end{split}
\end{equation}

We write
\begin{equation*}
\Big|1 + \sum_{i=1}^{2n}s_i\mu_i\Big|^{\al-1}=1+(\al-1)\sum_{i=1}^{2n}s_i\mu_i+O(|\mu|^2),
\end{equation*}
and decompose the multipliers $m_n$ accordingly. Thus let
\begin{equation}\label{cru9}
m_n^\al(\lambda_1,\ldots,\lambda_{2n+1}):=C_{n,\al}|\lambda_{2n+1}|^{\al-1}\lambda_{2n+1}\Big(\prod_{i=1}^{2n}\lambda_{i}\Big),
\end{equation}
\begin{equation}\label{cru9.1}
\begin{split}
m_n^1(\lambda_1,\ldots,\lambda_{2n+1})&:=C_{n,\al}|\lambda_{2n+1}|^{\al-1}\Big(\sum_{i=1}^{2n+1}\lambda_{i}\Big)\Big(\prod_{i=1}^{2n}\lambda_{i}\Big)\int_{[0,1]^{2n}}-\Big|\sum_{i=1}^{2n}s_i\mu_i\Big|^{\al-1}\,d\underline{s}\\
&=-C_{n,\al}\Big(\sum_{i=1}^{2n+1}\lambda_{i}\Big)\Big(\prod_{i=1}^{2n}\lambda_{i}\Big)\int_{[0,1]^{2n}}\Big|\sum_{i=1}^{2n}s_i\lambda_i\Big|^{\al-1}\,d\underline{s},
\end{split}
\end{equation}
\begin{equation}\label{cru9.2}
\begin{split}
m_n^{\al-1}(\lambda_1,\ldots,\lambda_{2n+1})&:=C_{n,\al}|\lambda_{2n+1}|^{\al-1}\Big(\sum_{i=1}^{2n}\lambda_{i}\Big)\Big(\prod_{i=1}^{2n}\lambda_{i}\Big)\\
&+C_{n,\al}|\lambda_{2n+1}|^{\al-1}\lambda_{2n+1}\Big(\prod_{i=1}^{2n}\lambda_{i}\Big)\int_{[0,1]^{2n}}(\al-1)\Big(\sum_{i=1}^{2n}s_i\mu_i\Big)\,d\underline{s}\\
&=\frac{(\al+1)C_{n,\al}}{2}|\lambda_{2n+1}|^{\al-1}\Big(\sum_{i=1}^{2n}\lambda_{i}\Big)\Big(\prod_{i=1}^{2n}\lambda_{i}\Big),
\end{split}
\end{equation}
\begin{equation}\label{cru9.3}
\begin{split}
m_n^{0}(\lambda_1,\ldots,\lambda_{2n+1})&=C_{n,\al}|\lambda_{2n+1}|^{\al-1}\Big(\prod_{i=1}^{2n}\lambda_{i}\Big)\Big\{\Big(\sum_{i=1}^{2n}\lambda_i\Big) \frac{\al-1}{2}\Big(\sum_{i=1}^{2n}\mu_{i}\Big)\\
&+\Big(\sum_{i=1}^{2n+1}\lambda_i\Big)\int_{[0,1]^{2n}}\Big(\Big|1 + \sum_{i=1}^{2n}s_i\mu_i\Big|^{\al-1}-1-(\al-1)\sum_{i=1}^{2n}s_i\mu_i\Big)\,d\underline{s}\Big\},
\end{split}
\end{equation}
where $\underline{s}=(s_1,\ldots,s_{2n})$.

Notice that
\begin{equation*}
m_n=m_n^\al+m_n^1+m_n^{\al-1}+m_n^0.
\end{equation*}
The multipliers $m_n^\al$, $m_n^1$, and $m_n^{\al-1}$ will essentially generate the symbols $\Sigma^\al$, $\Sigma^1$, and $\Sigma^{\al-1}$ in \eqref{plo2}. The multipliers $m_n^0$ will only add contributions in the error term $E$. 

{\bf{Step 3.}} As in \eqref{cru0}, for $\kappa\in\{\al,1,\al-1,0\}$ we define the multilinear operators $L_n^\kappa$ by
\begin{equation}\label{cru0.1}
\begin{split}
\mathcal{F}\{L_n^\kappa[f_1,\ldots,f_{2n},f_{2n+1}]\}(\xi)&:=\frac{i}{2n+1}\int_{\RR^{2n}} \widehat{f_1}(\eta_1)\ldots \widehat{f_{2n}}(\eta_{2n})\\
&\times \widehat{f_{2n+1}}\left(\xi-\sum_{i=1}^{2n}\eta_{i}\right)m_n^\kappa\left(\eta_1,\eta_2,\ldots,\eta_{2n},\xi-\sum_{i=1}^{2n}\eta_{i}\right) d\underline{\eta}.
\end{split}
\end{equation}
We examine the formula \eqref{cru20} and define, for $k\geq 4n+41$,
\begin{equation}\label{cru0.2}
\mathcal{L}_{n;k}^\kappa(t):=(2n+1)L_n^\kappa[P_{\leq k-4n-41}h(t),\ldots,P_{\leq k-4n-41}h(t),P_kh(t)].
\end{equation}
Clearly,
\begin{equation}\label{cru0.3}
\mathcal{N}_n=\sum_{\underline{k}\in X_n}\mathcal{N}_{n;\underline{k}}+\sum_{k\geq 4n+41}\{\mathcal{L}_{n;k}^0+\mathcal{L}_{n;k}^\alpha+\mathcal{L}_{n;k}^1+\mathcal{L}_{n;k}^{\alpha-1}\}.
\end{equation}

We have already seen, as a consequence of \eqref{cru13}, that the contribution of the terms $\mathcal{N}_{n;\underline{k}}$ can be incorporated into the error term $E$. The contribution of $\mathcal{L}_{n;k}^0$ can also be incorporated into this error term because
\begin{equation}\label{cru30}
\sum_{k\geq 4n+41}\big\{\|\mathcal{L}_{n;k}^0(t)\|_{H^{N_0\alpha}}+\|S\mathcal{L}_{n;k}^0(t)\|_{H^{N_1\alpha}}\big\}\lesssim_n\e_1^{2n+1}\langle t\rangle^{-1+p_0}.
\end{equation}
Indeed, these bounds follow as in the proof of \eqref{cru13}. We only need to notice that the symbols $m_n^0$ are homogeneous of order $2n+\al$ and do not lose derivatives, i.e.
\begin{equation*}
\big\|\mathcal{F}^{-1}[m_n^0(\la_1,\ldots,\la_{2n+1})\varphi_{k_1}(\la_1)\ldots\varphi_{k_{2n+1}}(\la_{2n+1})]\big\|_{L^1(\mathbb{R}^{2n+1})}\lesssim_n 2^{k_1+\ldots+k_{2n}}2^{2k_{\mathrm{sec}}}2^{(\al-2)k_{\mathrm{max}}},
\end{equation*} 
see \eqref{cru9.3}, where $k_{\mathrm{sec}}=\max(k_1,\ldots,k_{2n})$ and $k_{\mathrm{sec}}+2n+20\leq k_{2n+1}$.

{\bf{Step 4.}} We consider now the leading contributions, corresponding to the functions $\mathcal{L}_{n;k}^\al$. We define the symbols $\Sigma^\al_n$ and $B^{\al-1}_{n,1}$, by
\begin{equation}\label{cru31}
\begin{split}
\widetilde{\Sigma^\al_n}(\rho,\zeta)&:=2\pi C_{n,\al}\zeta|\zeta|^{\al-1}\varphi_{\geq 30}(\zeta)\widetilde{H_n}(\rho,\zeta),\\
\widetilde{B^{\al-1}_{n,1}}(\rho,\zeta)&:=2\pi C_{n,\al}\zeta|\zeta|^{\al-1}\varphi_{\geq 30}(\zeta)\frac{-\al\rho}{2\ze}\widetilde{H_n}(\rho,\zeta),
\end{split}
\end{equation} 
where, with $\underline{\mu}=(\mu_1,\ldots,\mu_{2n-1})$ and $g_n$ defined as in \eqref{Sig4},
\begin{equation}\label{cru32}
\widetilde{H_n}(\rho,\zeta):=\int_{\mathbb{R}^{2n-1}}\Big[\prod_{i=1}^{2n-1}\widetilde{g_n}(\mu_i,\zeta)\Big]\widetilde{g_n}\Big(\rho-\sum_{i=1}^{2n-1}\mu_i,\ze\Big)\,d\underline{\mu}.
\end{equation}
See also subsection \ref{ParaDiffCalc} for the notation related to paradifferential calculus. We notice that the symbols $\Sigma^\al_n$ are real-valued (thus the operators $T_{\Sigma_n^\al}$ are self-adjoint), but the symbols $B^{\al-1}_{n,1}$ are not real-valued. Notice also that $\sum_{n\geq 1}\Sigma_n^\al=\Sigma^\al$, compare with \eqref{Sig1}.

We define
\begin{equation}\label{cru35}
E^1_{n;k}:=\mathcal{L}^\al_{n,k}-iT_{\Sigma_n^\al}P_kh-iT_{B_{n,1}^{\al-1}}P_kh
\end{equation} 
and we will prove that $\sum_kE^1_{n,k}$ are acceptable errors, i.e.
\begin{equation}\label{cru36}
\sum_{k\geq 4n+41}\big\{\|E^1_{n;k}(t)\|_{H^{N_0\alpha}}+\|SE^1_{n;k}(t)\|_{H^{N_1\alpha}}\big\}\lesssim_n\e_1^{2n+1}\langle t\rangle^{-1+p_0}.
\end{equation}
Indeed, we start from the formula
\begin{equation}\label{cru40}
\widehat{\mathcal{L}^\al_{n,k}}(\xi)=iC_{n,\al}\int_{\RR^{2n}} \widehat{P_kh}(\xi-\rho)|\xi-\rho|^{\al-1}(\xi-\rho)\Big(\prod_{i=1}^{2n-1}\widehat{f}(\mu_i)\Big)\widehat{f}\Big(\rho-\sum_{i=1}^{2n-1}\mu_{i}\Big)\,d\rho d\underline{\mu},
\end{equation}
where $\widehat{f}(\eta):=\eta\widehat{h}(\eta)\varphi_{\leq k-4n-41}(\eta)$. This follows from the definitions \eqref{cru9}, \eqref{cru0.1}--\eqref{cru0.2}, and a change of variables. With $\zeta:=\xi-\rho/2$ we notice that we may replace the factor $|\xi-\rho|^{\al-1}(\xi-\rho)$ in \eqref{cru40} with $|\zeta|^{\al-1}\zeta[1-\al\rho/(2\zeta)]$ at the expense of an error $O(\rho^2/\xi^2)$, which leads to error terms satisfying \eqref{cru36}. Thus we may replace $\mathcal{L}^\al_{n,k}$ with $iT_{A_{n,k}}P_kh$, where
\begin{equation*}
\begin{split}
&\widetilde{A_{n,k}}(\rho,\zeta):=2\pi C_{n,\al}\zeta|\zeta|^{\al-1}\Big(1-\frac{\al\rho}{2\zeta}\Big)\varphi_{[k-4,k+4]}(\zeta)\widehat{H_{n,k}}(\rho),\\
&\widehat{H_{n,k}}(\rho):=\int_{\RR^{2n-1}} \Big(\prod_{i=1}^{2n-1}\widehat{f}(\mu_i)\Big)\widehat{f}\Big(\rho-\sum_{i=1}^{2n-1}\mu_{i}\Big)\, d\underline{\mu},\qquad \widehat{f}(\eta):=\eta\widehat{h}(\eta)\varphi_{\leq k-4n-41}(\eta).
\end{split}
\end{equation*}

Finally, we compare the symbols $A_{n,k}$ and $\Sigma_n^\al+B_{n,1}^{\al-1}$. Notice that if $k\geq 4n+41$ then
\begin{equation*}
\begin{split}
\widetilde{A_{n,k}}&(\rho,\zeta)-\varphi_{[k-4,k+4]}(\zeta)[\widetilde{\Sigma^\al_n}(\rho,\zeta)+\widetilde{B^{\al-1}_{n,1}}(\rho,\zeta)]=\widetilde{\Delta_{n,k}}(\rho,\zeta)\\
&:=2\pi C_{n,\al}\varphi_{[k-4,k+4]}(\zeta)\zeta|\zeta|^{\al-1}\Big(1-\frac{\al\rho}{2\zeta}\Big)[\widehat{H_{n,k}}(\rho)-\widetilde{H_n}(\rho,\zeta)].
\end{split}
\end{equation*}
It is easy to see that 
\begin{equation*}
\|\Delta_{n,k}\|_{\mathcal{L}^\infty_0}\lesssim_n\e_1^{2n}2^{-\beta k}\langle t\rangle^{-1},\qquad \|\Delta_{n,k}\|_{\mathcal{L}^2_0}+\|S_{x,\zeta}\Delta_{n,k}\|_{\mathcal{L}^2_0}\lesssim_n\e_1^{2n}2^{-\beta k}\langle t\rangle^{-1/2+p_0},
\end{equation*}
see definitions \eqref{nor1}, using the fact that $\widehat{H_{n,k}}(\rho)-\widetilde{H_n}(\rho,\zeta)$ is nontrivial only when at least one the factors of $h$ has frequency $\geq k-4n-50$. The bounds \eqref{LqBdTa} and the identities \eqref{Alu3} show that
\begin{equation*}
\sum_{k\geq 4n+41}\big\{\|T_{\Delta_{n,k}}P_k(t)\|_{H^{N_0\alpha}}+\|ST_{\Delta_{n,k}}P_k(t)\|_{H^{N_1\alpha}}\big\}\lesssim_n\e_1^{2n+1}\langle t\rangle^{-1+p_0}.
\end{equation*} 
This completes the proof of \eqref{cru36}.

{\bf{Step 5.}} We consider now the contributions corresponding to the functions $\mathcal{L}_{n;k}^1$. We define the symbols $\Sigma^1_n$ by
\begin{equation}\label{cru50}
\begin{split}
\widetilde{\Sigma^1_n}(\rho,\zeta):=-2\pi C_{n,\al}\zeta\varphi_{\geq 30}(\zeta)\int_{\mathbb{R}^{2n-1}}&p_n\Big(\underline{\mu},\rho-\sum_{i=1}^{2n-1}\mu_i\Big)\\
&\times\Big[\prod_{i=1}^{2n-1}\widetilde{g_n}(\mu_i,\zeta)\Big]\widetilde{g_n}\Big(\rho-\sum_{i=1}^{2n-1}\mu_i\Big)\,d\underline{\mu},
\end{split}
\end{equation}
where $\widetilde{g_n}$ are $p_n$ are defined as in \eqref{Sig4} and \eqref{Sig5}. The multipliers $p_n$ satisfy the properties summarized in Lemma \ref{mbounds} (iii). As in {\bf{Step 4}} it is not hard to see that the error terms $E^2_{n,k}:=\mathcal{L}^1_{n;k}-iT_{\Sigma^1_n}P_kh$ satisfy bounds similar to \eqref{cru36}. In other words, one can replace $\sum_{k\geq 4n+41}\mathcal{L}^1_{n;k}$ with $iT_{\Sigma^1_n}P_{\geq 4n+41}h$ at the expense of acceptable errors.

{\bf{Step 6.}} We consider the remaining contributions corresponding to the functions $\mathcal{L}_{n;k}^{\al-1}$. Let
\begin{equation}\label{cru41}
\widetilde{B^{\al-1}_{n,2}}(\rho,\zeta):=2\pi \frac{(\al+1)C_{n,\al}}{2}|\zeta|^{\al-1}\varphi_{\geq 30}(\zeta)\rho\widetilde{H_n}(\rho,\zeta),
\end{equation} 
where $H_n$ is as in \eqref{cru32}. As in {\bf{Step 4}} the error terms $E^3_{n,k}:=\mathcal{L}^{\al-1}_{n;k}-iT_{B^{\al-1}_{n,2}}P_kh$ satisfy bounds similar to \eqref{cru36}, so one can replace $\sum_{k\geq 4n+41}\mathcal{L}^1_{n;k}$ with $iT_{B^{\al-1}_{n,2}}P_{\geq 4n+41}h$ at the expense of acceptable errors. 

Finally, we notice that the symbols $B^{\al-1}_{n,1}$ (defined in \eqref{cru31}) and $B^{\al-1}_{n,2}$ combine to produce the symbols $\Sigma^{\al-1}$ defined in \eqref{Sig3}. The bounds \eqref{Sig8} follow directly from the definitions \eqref{nor1}, the pointwise bounds \eqref{prodecay0}, and the bootstrap assumptions \eqref{zxc4}. This completes the proof of the lemma. 
\end{proof}

\subsection{The normalized variable $h^\ast$}\label{hstar} We would like to use Lemma \ref{MainhDec} to prove energy estimates. To do this directly we would need the operator $T_\Sigma$ to be self-adjoint, which is equivalent to the symbol $\Sigma$ being real-valued. This is true for the symbols $\Sigma^\al$ and $\Sigma^1$, but not for the symbol $\Sigma^{\al-1}$. We correct this by introducing the so-called ``good variable'' $h^\ast$.

We define the symbol $Q$ by the formula
\begin{equation}\label{gri4}
Q(x,\zeta):=\varphi_{\geq 30}(\zeta)\sum_{n\geq 1} d_n(-1)^n(g_n(x,\zeta))^{2n}
\end{equation}
where the symbols $g_n$ are as in \eqref{Sig4}. Notice that
\begin{equation}\label{gri4.1}
\Sigma^{\al}(x,\ze)=\Lambda(\ze)Q(x,\ze),\qquad i\Sigma^{\al-1}=\frac{\Lambda(\ze)}{2\zeta}(\partial_x Q)(x,\zeta).
\end{equation}

\begin{lemma}\label{nok0}
Let
\begin{equation}\label{nok1}
b(x,\zeta,t):=(1+Q(x,\ze,t))^{1/(2\al)},\qquad h^\ast:=T_{b}h.
\end{equation}
Then we have
\begin{equation}\label{nok2}
\partial_t h^\ast=i\Lambda h^\ast+iT_{\Sigma^\al+\Sigma^1}h^\ast+E',
\end{equation}
where the symbols $\Sigma^\al$ and $\Sigma^1$ are defined as in \eqref{Sig1}--\eqref{Sig2}. The functions $h^\ast, E'\in C([0,T]:H^{N_0\al})$ satisfy the bounds
\begin{equation}\label{nok2.5}
\|h^\ast(t)\|_{H^{N_0\alpha}}+\|Sh^\ast(t)\|_{H^{N_1\alpha}}\lesssim\e_1\langle t\rangle^{p_0},\qquad \sup_{k\in\mathbb{Z}}\,(2^{k/2-\beta k}+2^{N_2\al k})\|P_kh^\ast(t)\|_{L^\infty}\lesssim \e_1\langle t\rangle^{-1/2},
\end{equation}
and
\begin{equation}\label{nok3}
\|E'(t)\|_{H^{N_0\alpha}}+\|SE'(t)\|_{H^{N_1\alpha}}\lesssim\e_1^2\langle t\rangle^{-1+p_0}.
\end{equation}
\end{lemma}

\begin{proof}
We apply $T_b$ to the equation \eqref{yui1}, so
\begin{equation}\label{nok4}
\partial_t h^\ast=[\partial_t,T_b]h+iT_{\Lambda+\Sigma}h^\ast+i[T_b,T_{\Lambda+\Sigma}]h+T_bE.
\end{equation}
Notice that $Q$, $b-1$, and $\partial_tb$ are symbols of order $0$, 
\begin{equation}\label{nok5}
\|b'(t)\|_{\mathcal{L}^\infty_0}\lesssim \e_1^2\langle t\rangle^{-1},\qquad\|b'(t)\|_{\mathcal{L}^2_0}+\|S_{x,\zeta}b'(t)\|_{\mathcal{L}^2_0}\lesssim \e_1^2\langle t\rangle^{-1/2+p_0},
\end{equation}
where $b'\in\{Q,b-1,\partial_tb\}$. This shows that $[\partial_t,T_b]h$ and $T_bE$ are acceptable error terms satisfying \eqref{nok3}. Moreover, the bounds \eqref{nok2.5} hold, using also \eqref{nok2.7}. 

To prove the desired identity \eqref{nok2} it suffices to show that the term
\begin{equation*}
iT_{\Sigma^{\al-1}}h^\ast+i[T_b,T_{\Lambda+\Sigma}]h
\end{equation*}
is an error type term satisfying \eqref{nok3}. In view of Lemma \ref{PropProd}, we may replace the expression above by 
\begin{equation}\label{gri1}
iT_{\Sigma^{\al-1}b+i\{b,\Lambda+\Sigma\}}h,
\end{equation}
at the expense of acceptable errors. Finally, we notice that the choice of the symbol $b$ in \eqref{nok1} is such that $\Sigma^{\al-1}b+i\{b,\Lambda+\Sigma\}$ is a symbol of order $0$, satisfying bounds similar to \eqref{nok5}, so the expression in \eqref{gri1} is again a suitable error term. The conclusion of the lemma follows.
\end{proof}

\subsection{The improved energy bounds}\label{inpren}

We can now prove the bounds in Proposition \ref{proEEZ}.

\begin{lemma}\label{hener}
With the assumptions in Proposition \ref{MainProp}, we have, for any $t\in[0,T]$,
\begin{equation}\label{Ens29}
\|h(t)\|_{H^{N_0 \al}}+\|Sh(t)\|_{H^{N_1 \al}}\lesssim \e_0\langle t\rangle^{p_0}.
\end{equation}
\end{lemma}

\begin{proof} We divide the proof in several steps. Let $B:=\Lambda+\Sigma^\al+\Sigma^1$.

{\bf{Step 1.}} We prove first suitable bounds for $h^\ast$, i.e.
\begin{equation}\label{Ens40}
\sum_{m=0}^{N_0}\|T_B^mh^\ast(t)\|_{L^2}+\sum_{m=0}^{N_1}\|T_B^mSh^\ast(t)\|_{L^2}\lesssim \e_0\langle t\rangle^{p_0}.
\end{equation}
For this we start from the identities \eqref{nok2}, written in the form $\partial_th^\ast=iT_Bh^\ast+E'$, recall that the operators $T_B^m$ are self-adjoint, and apply energy estimates,
\begin{equation}\label{Ens41}
\frac{d}{dt}\langle T_B^mh^\ast,T_B^mh^\ast\rangle=2\Re\langle [\partial_t,T_B^m]h^\ast,T_B^mh^\ast\rangle+2\Re\langle T_B^mE',T_B^mh^\ast\rangle.
\end{equation}
Notice that, for $m\in[0,N_0]$,
\begin{equation*}
\|T_B^m h^\ast(t)\|_{L^2}\lesssim \e_1\langle t\rangle^{p_0},\qquad\|[\partial_t,T_B^m]h^\ast(t)\|_{L^2}+\|T_B^mE'(t)\|_{L^2}\lesssim \e_1^2\langle t\rangle^{-1+p_0},
\end{equation*}
as a consequence of Lemma \ref{PropProd} (i), the bootstrap assumptions \eqref{zxc4}, and the symbol bounds. Thus the right-hand side of \eqref{Ens41} is dominated by $C\e_1^3\langle t\rangle^{-1+2p_0}$. We integrate in time and use the initial-time bound $\|T_B^mT_bh(0)\|_{L^2}\lesssim \e_0$ to prove the first inequality in \eqref{Ens40}. 

The proof of the weighted bound in \eqref{Ens40} is similar. Since $[S,\partial_t-i\Lambda]=(-\al)(\partial_t-i\Lambda)$, it follows from \eqref{nok2} that
\begin{equation*}
(\partial_t-i\Lambda)Sh^\ast=iT_{\Sigma^\al+\Sigma^1}(Sh^\ast)+i[S,T_{\Sigma^\al+\Sigma^1}]h^\ast+SE'+\al(iT_{\Sigma^\al+\Sigma^1}h^\ast+E').
\end{equation*}
As before, we apply the operators $T_B^m$, $m\in[0,N_1]$, and derive the identities
\begin{equation}\label{Ens42}
\begin{split}
\frac{d}{dt}\langle T_B^mSh^\ast,T_B^mSh^\ast\rangle&=2\Re\langle [\partial_t,T_B^m]Sh^\ast,T_B^mSh^\ast\rangle\\
&+2\Re\langle T_B^m[iT_{S_{x,\ze}(\Sigma^\al+\Sigma^1)}h^\ast+\al iT_{\Sigma^\al+\Sigma^1}h^\ast+SE'+\al E'], T_B^mSh^\ast\rangle.
\end{split}
\end{equation}
It follows from Lemma \ref{PropProd} (i) and the bootstrap assumptions \eqref{zxc4} that
\begin{equation*}
\|T_B^mSh^\ast(t)\|_{L^2}\lesssim \e_1\langle t\rangle^{p_0}
\end{equation*}
and
\begin{equation*}
\begin{split}
\|[\partial_t,T_B^m]Sh^\ast(t)\|_{L^2}&+\|T_B^mT_{S_{x,\ze}(\Sigma^\al+\Sigma^1)}h^\ast(t)\|_{L^2}+\|T_B^mT_{\Sigma^\al+\Sigma^1}h^\ast(t)\|_{L^2}\\
&+\|T_B^mSE'(t)\|_{L^2}+\|T_B^mE'(t)\|_{L^2}\lesssim \e_1^2\langle t\rangle^{-1+p_0},
\end{split}
\end{equation*}
for any $m\in[0,N_1]$. Thus the right-hand side of \eqref{Ens42} is dominated by $C\e_1^3\langle t\rangle^{-1+2p_0}$, and the second bound in \eqref{Ens40} follows by integration in time.

{\bf{Step 2.}} We can prove now the bounds \eqref{Ens29}, using just elliptic estimates. The bounds \eqref{Ens40}, together with the simple identities
\begin{equation*}
\Lambda^m=T_B^m+\{\Lambda^m-(\Lambda+T_{\Sigma^\al+\Sigma^1})^m\},
\end{equation*}
and the symbol bounds \eqref{Sig8}, show that
\begin{equation*}
\|h^\ast(t)\|_{H^{N_0\al}}+\|Sh^\ast(t)\|_{N^{N_1\al}}\lesssim \e_0\langle t\rangle^{p_0}.
\end{equation*}
The desired conclusions \eqref{Ens29} follow once we recall that $h=h^\ast-T_{b-1}h$ and use \eqref{nok5}. 
\end{proof}

\begin{lemma}\label{proEEv}
With the assumptions in Proposition \ref{MainProp}, we have, for any $t\in[0,T]$,
\begin{equation}\label{Ens22}
\|x\partial_xv(t)\|_{L^2}\lesssim \e_0\langle t\rangle^{p_0}.
\end{equation}
\end{lemma}

\begin{proof}
Notice that 
\begin{equation}\label{eqprof}
e^{it\Lambda} \partial_t v = (\partial_t - i\Lambda) h = \mathcal{N}=\sum_{n\geq 1}\mathcal{N}_n,
\end{equation}
using \eqref{Poi4} and the definition $v(t)=e^{-it\Lambda}h(t)$. Notice also that
\begin{equation}\label{StoV}
Sh = e^{it\Lambda}(x\partial_xv) + \alpha te^{it\Lambda}(\partial_tv),
\end{equation}
which is a simple consequence of the identity $h(t)=e^{it\Lambda}v(t)$. Since we have already proved that $\|Sh\|_{L^2}\lesssim\varep_0\langle t \rangle^{p_0}$, for \eqref{Ens22} it suffices to show that
\begin{equation}\label{estN'}
{\| \partial_tv(t)\|}_{L^2} \lesssim  \e_0\langle t \rangle^{-1+p_0}.
\end{equation}
This follows from \eqref{eqprof} and Lemma \ref{MainhDec}.
\end{proof}

\section{Modified scattering and pointwise decay}\label{Poi1}

In this section we prove the following:

\begin{proposition}\label{proZ}
With the assumptions in Proposition \ref{MainProp}, we have, for any $t\in[0,T]$,
\begin{equation}
\label{proZconc}
\sup_{t\in[0,T]} {\|v(t)\|}_Z \lesssim \e_0.
\end{equation}
\end{proposition}

The rest of the section is concerned with the proof of this proposition. The analysis is more subtle here, and we need to differentiate between the cubic nonlinearity $\mathcal{N}_1$, which contributes to modified scattering, and the remaining quintic and higher order terms.

\subsection{Modified scattering}\label{Nonl}

We rewrite $e^{-it\Lambda}\mathcal{N}_1(t)$ in terms of the profile $v$, 
\begin{equation}\label{nf47}
\begin{split}
\mathcal{F}(e^{-it\Lambda}\mathcal{N}_1)&(\xi,t)=i\int_{\R\times\R}m_1(\eta_1,\eta_2,\xi-\eta_1-\eta_2)\\
&\times e^{it(-\Lambda(\xi)+\Lambda(\eta_1)+\Lambda(\eta_2)+\Lambda(\xi-\eta_1-\eta_2))}\widehat{v}(\eta_1,t)\widehat{v}(\eta_2,t)\widehat{v}(\xi-\eta_1-\eta_2,t)\,d\eta_1 d\eta_2,
\end{split}
\end{equation}
using \eqref{nonl4} and the formula $h(t)=e^{it\Lambda}v(t)$. The formula \eqref{eqprof} becomes
\begin{equation}\label{nf48}
(\partial_t\widehat{v})(\xi,t)= iI(\xi,t)+ e^{-it\Lambda(\xi)}\what{\mathcal{R}_{\geq 2}}(\xi,t),
\end{equation}
where
\begin{equation}\label{Rdef}
iI(\xi,t):=\mathcal{F}(e^{-it\Lambda}\mathcal{N}_1)(\xi,t)\quad\text{ and }\quad\mathcal{R}_{\geq 2}=\sum_{n\geq 2}\mathcal{N}_n.
\end{equation}

In analyzing the formula \eqref{nf48}, the main contribution comes from the stationary points of the phase function $(t,\eta_1,\eta_2)\to t\Phi(\xi,\eta_1,\eta_2)$, where
\begin{equation}\label{nf49}
\Phi(\xi,\eta_1,\eta_2)
  := -\Lambda(\xi)+\Lambda(\eta_1)+\Lambda(\eta_2)+\Lambda(\xi-\eta_1-\eta_2).
\end{equation}
More precisely, one needs to understand the contribution of the {\it{spacetime resonances}},
i.e., the points where
\begin{align*}
 \Phi(\xi,\eta_1,\eta_2) = (\partial_{\eta_1}\Phi)(\xi,\eta_1,\eta_2)
  = (\partial_{\eta_2}\Phi)(\xi,\eta_1,\eta_2) = 0.
\end{align*}
In our case, it is easy to see that the only spacetime resonances correspond to $(\xi,\eta_1,\eta_2)\in\{(\xi,\xi,\xi),(\xi,\xi,-\xi),(\xi,-\xi,\xi)\}$.
Moreover, the contribution from these points is not absolutely integrable in time,
and we have to identify and eliminate its leading order term using a suitable logarithmic phase correction.
More precisely, we define
\begin{equation}\label{nf50}
\begin{split}
& \widetilde{c}(\xi) :=-\mathcal{K}_\alpha|\xi|^{2-\alpha}[m_1(\xi,\xi,-\xi)+m_1(\xi,-\xi,\xi)+m_1(-\xi,\xi,\xi)],\qquad\mathcal{K}_\alpha:=\frac{2\pi}{\gamma\alpha(\alpha-1)},
\\
& L(\xi,t) := \widetilde{c}(\xi)\int_0^t {|\what{v}(\xi,s)|}^2 \frac{1}{s+1}\,ds.
\end{split}
\end{equation}

The formula \eqref{nf48} then becomes
\begin{equation}\label{nf51}
\begin{split}
\frac{d}{dt}[\widehat{v}(\xi,t)e^{iL(\xi,t)}]= ie^{iL(\xi,t)}
  \Big[I(\xi,t) + \widetilde{c}(\xi)\frac{|\what{v}(\xi,t)|^2}{t+1} \what{v}(\xi,t) \Big]+e^{-it\Lambda(\xi)} e^{iL(\xi,t)} \what{\mathcal{R}_{\geq 2}}(\xi,t).
\end{split}
\end{equation}
Notice that the phase $L$ is real-valued.
Therefore, to complete the proof of Proposition \ref{proZ}, it suffices to prove the following main lemma:

\begin{lemma}\label{mainlem}
For any $m \in \{1,2,\ldots\}$ and any $t_1 \leq t_2 \in [2^{m}-2,2^{m+1}] \cap [0,T]$,
we have
\begin{equation}
\label{Zcontrolconc}
{\big\| (|\xi|^{1/2+\beta} + |\xi|^{N_2\alpha+1}) [\widehat{v}(\xi,t_2)e^{iL(\xi,t_2)}-\widehat{v}(\xi,t_1)e^{iL(\xi,t_1)}] \big\|}_{L^\infty_\xi} \lesssim \e_0 2^{-p_0m}.
\end{equation}
\end{lemma}

\subsection{Proof of Lemma \ref{mainlem}}\label{secprmainlem}

In this subsection we provide the proof of the more technical Lemma \ref{mainlem}. We first notice that the desired conclusion follows easily for large and small enough frequencies.
Indeed, for any $t \in [2^{m}-2,2^{m+1}] \cap[0,T]$ and $|\xi| \approx 2^k$ with $k \in \mathbb{Z}$ and
\begin{align*}
k \in (-\infty, -10(p_0/\beta)m] \cup [5p_0m-1,\infty) ,
\end{align*}
we can use the interpolation inequality \eqref{interp1}, the bounds \eqref{Ens2}, and the assumption $N_0/2 \geq N_2+1$ to obtain
\begin{equation*}
\begin{split}
\big( |\xi|^{1/2+\beta} + &|\xi|^{N_2\al+1} \big) |\widehat{P_kv}(\xi,t)|\\
& \lesssim (2^{k(1/2+\beta)} + 2^{(N_2\al+1)k}) \big[ 2^{-k} {\|\what{P'_kv}\|}_{L^2} \big(2^k {\|\partial \what{P'_kv} \|}_{L^2}
  + {\|\what{P'_kv}\|}_{L^2} \big) \big]^{1/2}
  \\
& \lesssim \e_0 \langle t \rangle^{p_0} \min \big( 2^{\beta k} , 2^{-k/2} \big) ,
\\
& \lesssim \e_0\langle t \rangle^{-p_0}.
\end{split}
\end{equation*}

It remains to prove \eqref{Zcontrolconc} in the intermediate range $|\xi| \in [2^{-10mp_0/\beta},2^{5p_0m-1}]$.
For $k\in\mathbb{Z}$ let $v_k:=P_k v$. For any $k_1,k_2,k_3 \in \mathbb{Z}$,  let
\begin{align}
\label{rak0}
\begin{split}
I_{k_1,k_2,k_3}(\xi,t):=&\int_{\R\times\R}m_1(\eta_1,\eta_2,\xi-\eta_1-\eta_2)\\
&\times e^{it\Phi(\xi,\eta_1,\eta_2)}\widehat{v_{k_1}}(\eta_1,t)\widehat{v_{k_2}}(\eta_2,t)\widehat{v_{k_3}}(\xi-\eta_1-\eta_2,t)\,d\eta_1 d\eta_2.
\end{split}
\end{align}
Using \eqref{zxc4} and Lemma \ref{Poi2} we know that for any $t\in[0,T]$, $t'\in[t/2,2t]$, and $l \leq 0$
\begin{align}
\label{boundsflow}
\begin{split}
{\|\what{v_l}(t)\|}_{L^2} + 2^l {\|\partial\what{v_l}(t) \|}_{L^2}
  & \lesssim \e_1 \langle t \rangle^{p_0},
\\
{\|e^{it'\Lambda}v_l(t)\|}_{L^\infty} & \lesssim \e_1 2^{-l/2+2\beta l} \langle t \rangle^{-1/2},
\\
{\|\what{v_l}(t) \|}_{L^\infty} & \lesssim \e_1 2^{-l(1/2+\beta)} ,
\end{split}
\end{align}
whereas, for $l\geq 0$, 
\begin{align}
\label{boundsfhigh}
\begin{split}
{\|\what{v_l}(t)\|}_{L^2} & \lesssim \e_1 2^{-N_0\al l} \langle t \rangle^{p_0},
\\
{\|\what{v_l}(t)\|}_{L^2} + 2^l {\|\partial\what{v_l}(t) \|}_{L^2}
  & \lesssim \e_1 \langle t \rangle^{p_0},
\\
{\|e^{it'\Lambda}v_l(t)\|}_{L^\infty} & \lesssim \e_1 2^{-N_2\al l} \langle t \rangle^{-1/2},
\\
{\|\what{v_l}(t) \|}_{L^\infty} & \lesssim \e_1 2^{-(N_2\al+1) l} .
\end{split}
\end{align}

Since $h$ is real-valued, we have $\overline{\widehat{v}(\xi,t)}=\widehat{v}(-\xi,t)$. In view of \eqref{nf51}, to complete the proof of Lemma \ref{mainlem} it suffices to prove the following:

\begin{lemma}\label{mainlemma2}
Assume that $k \in [-10(p_0/\beta)m, 5p_0m]$, $|\xi| \in [2^k,2^{k+1}]$,
$m \geq 1/(100p_0)$, $t_1\leq t_2 \in [2^{m},2^{m+1}] \cap [0,T]$, and $k_1,k_2,k_3\in\mathbb{Z}$. Then
\begin{equation}
\label{mainlemma2conc1}
\begin{split}
\Big|\int_{t_1}^{t_2} e^{iL(\xi,s)} \Big[ I_{k_1,k_2,k_3}(\xi,s)
  &-\mathcal{K}_\alpha|\xi|^{2-\alpha}\frac{m_1(\xi,\xi,-\xi)\widehat{v_{k_1}}(\xi,s)\widehat{v_{k_2}}(\xi,s)\widehat{v_{k_3}}(-\xi,s)}{s+1}\\
	&-\mathcal{K}_\alpha|\xi|^{2-\alpha}\frac{m_1(\xi,-\xi,\xi)\widehat{v_{k_1}}(\xi,s)\widehat{v_{k_2}}(-\xi,s)\widehat{v_{k_3}}(\xi,s)}{s+1}\\
	&-\mathcal{K}_\alpha|\xi|^{2-\alpha}\frac{m_1(-\xi,\xi,\xi)\widehat{v_{k_1}}(-\xi,s)\widehat{v_{k_2}}(\xi,s)\widehat{v_{k_3}}(\xi,s)}{s+1}\Big]\,ds\Big|\\
	&\lesssim \e_1^32^{-20N_2p_0m}2^{-p_0\max(|k_1|,|k_2|,|k_3|)}.
\end{split}
\end{equation}
Moreover
\begin{equation}
\label{mainlemma2conc3}
\Big|\int_{t_1}^{t_2}e^{iL(\xi,s)}e^{-is\Lambda(\xi)} \what{\mathcal{R}_{\geq 2}}(\xi,s)\,ds\Big| \lesssim \e_0 2^{-20 N_2p_0m}.
\end{equation}
\end{lemma}

\subsubsection{Proof of \eqref{mainlemma2conc1}}  We start with the cubic bound and consider several cases.

\begin{lemma}\label{lemma1}
The bound \eqref{mainlemma2conc1} holds provided that
\begin{equation}\label{vcond1}
\max (k_1,k_2,k_3)\geq 2m/N_0\qquad\text{ or }\qquad \min(k_1,k_2,k_3)\leq -4m/7.
\end{equation}
\end{lemma}

\begin{proof}
For $s\approx 2^m$ we estimate, using Lemma \ref{touse} (iii) and \eqref{nonl7},
\begin{equation}\label{poi30}
\begin{split}
\|&I_{k_1,k_2,k_3}(s)\|_{L^\infty}\lesssim \|\mathcal{F}^{-1}\{e^{is\Lambda}I_{k_1,k_2,k_3}(s)\}\|_{L^1}\\
&\lesssim 2^{\min(k_1,k_2,k_3)}2^{3\max(k_1,k_2,k_3,0)}\|e^{is\Lambda}v_{k_1}(s)\|_{L^{p_1}}\|e^{is\Lambda}v_{k_2}(s)\|_{L^{p_2}}\|e^{is\Lambda}v_{k_3}(s)\|_{L^{p_3}},
\end{split}
\end{equation}
for any choice of $p_1,p_2,p_3\in\{2,\infty\}$ satisfying $1/p_1+1/p_2+1/p_3=1$. Assume, without loss of generality, that $k_1\leq k_2\leq k_3$. We set $p_3=\infty$, $p_1=p_2=2$. Using \eqref{boundsflow}--\eqref{boundsfhigh} and noticing that we may assume $k_3\geq k-10$, the right-hand side of \eqref{poi30} is dominated by
\begin{equation*}
C\e_1^22^{2p_0m}2^{k_1}2^{3k_3^+}\|e^{is\Lambda}v_{k_3}(s)\|_{L^{\infty}}\lesssim \e_1^32^{11(p_0/\beta)m}2^{k_1}2^{3k_3^+}\min(2^{-(N_0-1)k_3^+},2^{-m/2}2^{-N_2k_3^+}).
\end{equation*}
Assuming that \eqref{vcond1} holds, it follows that
\begin{equation*}
\|I_{k_1,k_2,k_3}(s)\|_{L^\infty}\lesssim \e_1^32^{-1.01m}2^{-2p_0k_3^+}2^{p_0k_1}.
\end{equation*}
The contribution of the other terms, which contain the symbol $m_1$, can be estimated easily using the last bounds in \eqref{boundsflow}--\eqref{boundsfhigh}. This suffices to prove the desired conclusion \eqref{mainlemma2conc1} in this case.
\end{proof}

\begin{lemma}\label{lemma2}
The bound \eqref{mainlemma2conc1} holds provided that
\begin{align}
\label{lemma2cond}
\begin{split}
& k_1,k_2,k_3\in[-4m/7,2m/N_0],
\\
& \max (|k_1-k_2|,|k_1 - k_3| , |k_2-k_3|) \geq 8.
\end{split}
\end{align}
\end{lemma}

\begin{proof}
In this case we will show the stronger bound
\begin{equation}
\label{lemma2conc}
|I_{k_1,k_2,k_3}(\xi,s)|\lesssim \e_1^3 2^{-m} 2^{-200p_0m}.
\end{equation}
Without loss of generality (using changes of variables) we may assume that $|k_2-k_3| \geq 8$, $\max(k_2,k_3)\geq k-20$. Then
\begin{align}
\label{lowerbound1}
 |(\partial_{\eta_2}\Phi)(\xi,\eta_1,\eta_2)| = |\Lambda'(\eta_2) - \Lambda'(\xi-\eta_1-\eta_2)|\gtrsim 2^{(\alpha-1)\max(k_2,k_3)}
\end{align}
in the support of the integral. Therefore we can integrate by parts in $\eta_2$ in the integral expression \eqref{rak0} for $I_{k_1,k_2,k_3}$.
This gives
\begin{align*}
|I_{k_1,k_2,k_3}(\xi,s)| & \lesssim  |K_1(\xi,s)| + |K_2(\xi,s)| + |K_3(\xi,s)|,
\end{align*}
where
\begin{align}
\label{IBPeta}
\begin{split}
K_1(\xi) & := \int_{\mathbb{R}\times\mathbb{R}} e^{is\Phi(\xi,\eta_1,\eta_2)}
   n_1(\xi,\eta_1,\eta_2)
  \what{v_{k_1}}(\eta_1) (\partial\what{v_{k_2}})(\eta_2) \what{v_{k_3}}(\xi-\eta_1-\eta_2) \,d\eta_1 d\eta_2,
\\
K_2(\xi) & := \int_{\mathbb{R}\times\mathbb{R}} e^{is\Phi(\xi,\eta_1,\eta_2)}
   n_1(\xi,\eta_1,\eta_2)
  \what{v_{k_1}}(\eta_1) \what{v_{k_2}}(\eta_2) (\partial\what{v_{k_3}})(\xi-\eta_1-\eta_2) \,d\eta_1 d\eta_2,
\\
K_3(\xi) & := \int_{\mathbb{R}\times\mathbb{R}} e^{is\Phi(\xi,\eta_1,\eta_2)}
   (\partial_{\eta_2}n_1)(\xi,\eta_1,\eta_2)
  \what{v_{k_1}}(\eta_1) \what{v_{k_2}}(\eta_2) \what{v_{k_3}}(\xi-\eta_1-\eta_2) \,d\eta_1 d\eta_2,
\end{split}
\end{align}
and
\begin{align}
\label{symIBPeta}
 n_1(\xi,\eta_1,\eta_2) := \frac{m_1(\eta_1,\eta_2,\xi-\eta_1-\eta_2)}{s(\Lambda'(\eta_2) - \Lambda'(\xi-\eta_1-\eta_2)) } \varphi'_{k_1}(\eta_1)\varphi'_{k_2}(\eta_2)\varphi'_{k_3}(\xi-\eta_1-\eta_2).
\end{align}

Using \eqref{nonl7} and \eqref{lowerbound1} it is easy to see that
\begin{align}
\label{symIBPetaest1}
{\| n_1\|}_{S^\infty} \lesssim 2^{-m} 2^{\min(k_1,k_2,k_3)}2^{3\max(k_1,k_2,k_3,0)}2^{-(\alpha-1)\max(k_2,k_3)}.
\end{align}
We can estimate $K_1$ and $K_2$ as in \eqref{poi30}, using \eqref{boundsflow}-\eqref{boundsfhigh} and the bound $2^{-\max(k_2,k_3)}\lesssim 2^{10(p_0/\beta)m}$,
\begin{equation*}
\begin{split}
|K_1(\xi)| &\lesssim {\| n_1\|}_{S^\infty} 
  {\|v_{k_1}\|}_{L^2}  {\| \partial\what{v_{k_2}}\|}_{L^2}{\| e^{is\Lambda} v_{k_3}\|}_{L^\infty}
\\
& \lesssim 2^{-m} 2^{\min(k_1,k_2,k_3)}2^{3\max(k_1,k_2,k_3,0)}2^{10(p_0/\beta)m} \cdot \e_12^{p_0m}\cdot 
  \e_1 2^{-k_2} 2^{p_0m} \cdot \e_1 2^{-m/2}2^{-k_3}
\\
& \lesssim \e_1^3 2^{-1.01m}.
\end{split}
\end{equation*}
A similar estimate shows that $|K_2(\xi)|\lesssim \e_1^3 2^{-1.01m}$. Moreover, using also \eqref{nonl7.5},
\begin{equation*}
{\| \partial_{\eta_2}n_1\|}_{S^\infty} \lesssim 2^{-m} 2^{3\max(k_1,k_2,k_3,0)}2^{-(\alpha-1)\max(k_2,k_3)},
\end{equation*}
and $|K_3(\xi)|$ can be bounded similarly. This completes the proof of \eqref{lemma2conc} and the lemma.
\end{proof}

It remains to consider the case
\begin{align}
\label{poi40}
\begin{split}
&k \in [-10(p_0/\beta)m, 5p_0m],\qquad k_1,k_2,k_3\in[k-20,2m/N_0],\\
& \max (|k_1-k_2|,|k_1 - k_3| , |k_2-k_3|) \leq 7.
\end{split}
\end{align}
Without loss of generality, in proving \eqref{mainlemma2conc1} we may assume that $\xi>0$, $\xi\in[2^k,2^{k+1}]$. We decompose the integrals $I_{k_1,k_2,k_3}$ as
\begin{equation}\label{poi41}
\begin{split}
&I_{k_1,k_2,k_3}=\sum_{\iota_1,\iota_2,\iota_3\in\{+,-\}}I_{k_1,k_2,k_3}^{\iota_1,\iota_2,\iota_3},\\
&I_{k_1,k_2,k_3}^{\iota_1,\iota_2,\iota_3}(\xi):=\int_{\R\times\R}m_1(\eta_1,\eta_2,\xi-\eta_1-\eta_2)e^{it\Phi(\xi,\eta_1,\eta_2)}\widehat{v_{k_1}^{\iota_1}}(\eta_1)\widehat{v_{k_2}^{\iota_2}}(\eta_2)\widehat{v_{k_3}^{\iota_3}}(\xi-\eta_1-\eta_2)\,d\eta_1 d\eta_2,
\end{split}
\end{equation}
where $\widehat{v_l^\iota}(\mu):=\widehat{v_l}(\mu)\mathbf{1}_{\iota}(\mu)$, $1_{+}:=\mathbf{1}_{[0,\infty)}$, $1_{-}:=\mathbf{1}_{(-\infty,0]}$. Notice that $I_{k_1,k_2,k_3}^{-,-,-}(\xi)=0$. We estimate the remaining contributions in the next two lemmas.

\begin{lemma}\label{lemma3}
With the hypothesis of Lemma \ref{mainlemma2}, and assuming that \eqref{poi40} holds, we have
\begin{equation*}
\Big|\int_{t_1}^{t_2} e^{iL(\xi,s)} I_{k_1,k_2,k_3}^{\iota_1,\iota_2,\iota_3}(\xi,s)\,ds\Big|\lesssim \e_1^32^{-m/100},
\end{equation*}
for $(\iota_1,\iota_2,\iota_3)\in\{(+,+,+),(+,-,-), (-,+,-), (-,-,+)\}$.
\end{lemma}

\begin{proof} We will only prove the bound for the integral $I_{k_1,k_2,k_3}^{-,-,+}$, since the other bounds are similar. We integrate by parts in $s$. The main observation is that
\begin{equation}\label{poi42}
(a+b+c)^\alpha-a^\alpha-b^\alpha-c^\alpha\geq b[(a+b+c)^{\alpha-1}-b^{\alpha-1}]\gtrsim ba^{\alpha-1},
\end{equation}
if $a\geq b\geq c\in(0,\infty)$. Therefore
\begin{equation}\label{poi43}
|\Phi(\xi,\eta_1,\eta_2)|\approx 2^{\alpha k_3}
\end{equation}
in the support of the integral defining $I_{k_1,k_2,k_3}^{-,-,+}(\xi,s)$. Due to this lower bound we can integrate by parts in $s$ to obtain
\begin{align}
\label{IBPs}
\begin{split}
\Big| \int_{t_1}^{t_2} & e^{iL(\xi,s)} I_{k_1,k_2,k_3}^{-,-,+}(\xi,s) \, ds \Big|
  \lesssim  | N_1 (\xi,t_1) | + | N_1 (\xi,t_2) |
  \\ & + \int_{t_1}^{t_2} | N_2(\xi,s)|+ | N_3(\xi,s) | + | N_4(\xi,s)| + |(\partial_sL)(\xi,s)|| N_1(\xi,s) | \, ds,
\end{split}
\end{align}
where
\begin{align*}
\begin{split}
& N_1(\xi) := \int_{\R\times\R} e^{is\Phi(\xi,\eta_1,\eta_2)} \frac{m_1(\eta_1,\eta_2,\xi-\eta_1-\eta_2)}{\Phi(\xi,\eta_1,\eta_2)}
  \what{v_{k_1}^{-}}(\eta_1) \what{v_{k_2}^{-}}(\eta_2) \what{v_{k_3}^{+}}(\xi-\eta_1-\eta_2) \,d\eta_1 d\eta_2 ,
\\
& N_2(\xi) := \int_{\R\times\R} e^{is\Phi(\xi,\eta_1,\eta_2)} \frac{m_1(\eta_1,\eta_2,\xi-\eta_1-\eta_2)}{\Phi(\xi,\eta_1,\eta_2)}
  (\partial_s\what{v_{k_1}^{-}})(\eta_1) \what{v_{k_2}^{-}}(\eta_2) \what{v_{k_3}^{+}}(\xi-\eta_1-\eta_2) \,d\eta_1 d\eta_2 ,
\\
& N_3(\xi) := \int_{\R\times\R} e^{is\Phi(\xi,\eta_1,\eta_2)} \frac{m_1(\eta_1,\eta_2,\xi-\eta_1-\eta_2)}{\Phi(\xi,\eta_1,\eta_2)}
  \what{v_{k_1}^{-}}(\eta_1) (\partial_s\what{v_{k_2}^{-}})(\eta_2) \what{v_{k_3}^{+}}(\xi-\eta_1-\eta_2) \,d\eta_1 d\eta_2 ,
\\
& N_4(\xi) := \int_{\R\times\R} e^{is\Phi(\xi,\eta_1,\eta_2)} \frac{m_1(\eta_1,\eta_2,\xi-\eta_1-\eta_2)}{\Phi(\xi,\eta_1,\eta_2)}
  \what{v_{k_1}^{-}}(\eta_1) \what{v_{k_2}^{-}}(\eta_2) (\partial_s\what{v_{k_3}^{+}})(\xi-\eta_1-\eta_2) \,d\eta_1 d\eta_2 .
\end{split}
\end{align*}

In view of \eqref{poi43} we may insert a factor of $\varphi_{[-C,C]}(2^{-\alpha k_3}\Phi(\xi,\eta_1,\eta_2))$ in the integrals above, for some constant $C=C_\alpha$. Let $\psi$ denote the inverse Fourier transform of $\varphi_{[-C,C]}(x)/x$, so
\begin{equation}\label{poi45}
\frac{\varphi_{[-C,C]}(2^{-\alpha k_3}\Phi(\xi,\eta_1,\eta_2))}{\Phi(\xi,\eta_1,\eta_2)}=2^{-\alpha k_3}\int_{\mathbb{R}}\psi(\lambda)e^{-i\lambda 2^{-\alpha k_3}\Phi(\xi,\eta_1,\eta_2)}\,d\lambda.
\end{equation}
It follows that
\begin{equation*}
\begin{split}
N_1(\xi) = 2^{-\alpha k_3}\int_{\mathbb{R}}\psi(\lambda)\int_{\R\times\R} &e^{i(s-\lambda 2^{-\alpha k_3})\Phi(\xi,\eta_1,\eta_2)} m_1(\eta_1,\eta_2,\xi-\eta_1-\eta_2)\\
&\times\what{v_{k_1}^{-}}(\eta_1) \what{v_{k_2}^{-}}(\eta_2) \what{v_{k_3}^{+}}(\xi-\eta_1-\eta_2) \,d\eta_1 d\eta_2.
\end{split}
\end{equation*}
We use the $L^2\times L^2\times L^\infty$ estimate, as in \eqref{poi30}, for $|\lambda|\leq 2^{m/4}$ and the rapid decay of the function $\psi$ for $|\lambda|\geq 2^{m/4}$. Recalling also the constraints \eqref{poi40}, it follows that, for any $s\in[t_1,t_2]$,
\begin{equation*}
|N_1(\xi,s)|\lesssim \e_1^32^{-m/4}.
\end{equation*}
Similar estimates, using also the $L^2$ bounds $\|\partial_s \widehat{v_k}\|_{L^2}\lesssim \e_02^{-m+p_0m}$, see \eqref{estN'}, show that
\begin{equation*}
|N_2(\xi,s)|+|N_3(\xi,s)|+|N_4(\xi,s)|\lesssim \e_1^32^{-5m/4}.
\end{equation*}
Finally, using the definition \eqref{nf50},
\begin{equation*}
|(\partial_sL)(\xi,s)|\lesssim 2^{-m+100(p_0/\beta)m}.
\end{equation*}
The desired bound follows from \eqref{IBPs}.
\end{proof}

\begin{lemma}\label{lemma4}
With the hypothesis of Lemma \ref{mainlemma2}, and assuming that \eqref{poi40} holds, we have
\begin{equation}\label{poi50}
\begin{split}
&\int_{t_1}^{t_2}\Big|I_{k_1,k_2,k_3}^{+,+,-}(\xi,s)-\mathcal{K}_\alpha|\xi|^{2-\alpha}\frac{m_1(\xi,\xi,-\xi)\widehat{v_{k_1}}(\xi,s)\widehat{v_{k_2}}(\xi,s)\widehat{v_{k_3}}(-\xi,s)}{s+1}\Big|\,ds\lesssim \e_1^32^{-m/100},\\
&\int_{t_1}^{t_2}\Big|I_{k_1,k_2,k_3}^{+,-,+}(\xi,s)-\mathcal{K}_\alpha|\xi|^{2-\alpha}\frac{m_1(\xi,-\xi,\xi)\widehat{v_{k_1}}(\xi,s)\widehat{v_{k_2}}(-\xi,s)\widehat{v_{k_3}}(\xi,s)}{s+1}\Big|\,ds\lesssim \e_1^32^{-m/100},\\
&\int_{t_1}^{t_2}\Big|I_{k_1,k_2,k_3}^{-,+,+}(\xi,s)-\mathcal{K}_\alpha|\xi|^{2-\alpha}\frac{m_1(-\xi,\xi,\xi)\widehat{v_{k_1}}(-\xi,s)\widehat{v_{k_2}}(\xi,s)\widehat{v_{k_3}}(\xi,s)}{s+1}\Big|\,ds\lesssim \e_1^32^{-m/100}.
\end{split}
\end{equation}
\end{lemma}

\begin{proof} We prove only the bound on $I_{k_1,k_2,k_3}^{+,+,-}$, since the other two bounds are similar. We examine the integral defining $I_{k_1,k_2,k_3}^{+,+,-}(\xi,s)$, see \eqref{poi41}; we would like to show that the main contribution in this integral comes from a suitable neighborhood of the point $(\eta_1,\eta_2)=(\xi,\xi)$.

Let $\overline{l}$ denote the smallest integer with the property that $\overline{l}\geq -9m/20$. For integers $l\geq \overline{l}$ we define the functions $\varphi_l^{(\overline{l})}$ by $\varphi_l^{(\overline{l})}=\varphi_l$ if $l\geq\overline{l}+1$ and $\varphi_l^{(\overline{l})}=\varphi_{\leq l}$ if $l=\overline{l}$. Then we decompose
\begin{equation}\label{poi51}
I_{k_1,k_2,k_3}^{+,+,-}=\sum_{l_1,l_2\in[\overline{l},k_3+40]}J_{l_1,l_2}^{+,+,-},
\end{equation}
\begin{equation}\label{poi52}
\begin{split}
J_{l_1,l_2}^{+,+,-}(\xi):=\int_{\R\times\R}m_1(\eta_1,\eta_2,\xi-\eta_1-\eta_2)&e^{it\Phi(\xi,\eta_1,\eta_2)}\varphi_{l_1}^{(\overline{l})}(\xi-\eta_1)\varphi_{l_2}^{(\overline{l})}(\xi-\eta_2)\\
&\times\widehat{v_{k_1}^{+}}(\eta_1)\widehat{v_{k_2}^{+}}(\eta_2)\widehat{v_{k_3}^{-}}(\xi-\eta_1-\eta_2)\,d\eta_1 d\eta_2.
\end{split}
\end{equation}
For \eqref{poi50} it suffices to show that
\begin{equation}\label{poi53}
\Big|J_{\overline{l},\overline{l}}^{+,+,-}(\xi,s)-\mathcal{K}_\alpha|\xi|^{2-\alpha}\frac{m_1(\xi,\xi,-\xi)\widehat{v_{k_1}^+}(\xi,s)\widehat{v_{k_2}^+}(\xi,s)\widehat{v_{k_3}^-}(-\xi,s)}{s+1}\Big|\lesssim \e_1^32^{-m-m/99},
\end{equation}
and
\begin{equation}\label{poi54}
J_{l_1,l_2}^{+,+,-}(\xi,s)\lesssim \e_1^32^{-m-m/99}\qquad \text{ if }l_1\geq \overline{l}+1\text{ or }l_2\geq \overline{l}+1.
\end{equation}

{\it{Proof of \eqref{poi53}.}} Since $k\in[-10(p_0/\beta)m,10p_0m-1]$ and $2^l\|\partial\widehat{v_l}\|_{L^2}\lesssim \e_12^{p_0m}$, we have
\begin{equation*}
|\widehat{v_{k_1}^{+}}(\eta_1)-\widehat{v_{k_1}^{+}}(\xi)|+|\widehat{v_{k_2}^{+}}(\eta_2)-\widehat{v_{k_1}^{+}}(\xi)|+|\widehat{v_{k_3}^{-}}(\xi-\eta_1-\eta_2)-\widehat{v_{k_3}^{-}}(-\xi)|\lesssim \e_12^{11(p_0/\beta)m}2^{\overline{l}/2}
\end{equation*}
in the support of the integral defining $J_{\overline{l},\overline{l}}^{+,+,-}(\xi,s)$. Moreover, using \eqref{nonl7.5},
\begin{equation*}
|m_1(\eta_1,\eta_2,\xi-\eta_1-\eta_2)-m_1(\xi,\xi,-\xi)|\lesssim 2^{\overline{l}}2^{200p_0m}.
\end{equation*}
Therefore
\begin{equation}\label{poi55}
\begin{split}
\Big|J_{\overline{l},\overline{l}}^{+,+,-}(\xi,s)&-m_1(\xi,\xi,-\xi)\widehat{v_{k_1}^+}(\xi,s)\widehat{v_{k_2}^+}(\xi,s)\widehat{v_{k_3}^-}(-\xi,s)\\
&\times\int_{\R\times\R}e^{is\Phi(\xi,\eta_1,\eta_2)}\varphi_{\leq \overline{l}}(\xi-\eta_1)\varphi_{\leq\overline{l}}(\xi-\eta_2)\,d\eta_1 d\eta_2\Big|\lesssim 2^{5\overline{l}/2}2^{40(p_0/\beta)m}.
\end{split}
\end{equation}

For \eqref{poi53} it remains to prove that if $s\in [2^{m},2^{m+1}]$ then 
\begin{equation}\label{poi56}
\Big|\int_{\R\times\R}e^{is\Phi(\xi,\eta_1,\eta_2)}\varphi_{\leq \overline{l}}(\xi-\eta_1)\varphi_{\leq\overline{l}}(\xi-\eta_2)\,d\eta_1 d\eta_2-\frac{\mathcal{K}_\alpha|\xi|^{2-\alpha}}{s}\Big|\lesssim 2^{5\overline{l}/2}2^{40(p_0/\beta)m}.
\end{equation}
For this we make the change of variables $\eta_1=\xi+x_1$, $\eta_2=\xi+x_2$. Recalling that $\Lambda(\mu)=\gamma\mu|\mu|^{\alpha-1}$ we notice that
\begin{equation*}
\begin{split}
\Phi(\xi,\eta_1,\eta_2)&=-\Lambda(\xi)+\Lambda(\xi+x_1)+\Lambda(\xi+x_2)-\Lambda(\xi+x_1+x_2)\\
&=-\frac{\gamma\alpha(\alpha-1)x_1x_2}{\xi^{2-\alpha}}+O(2^{3\overline{l}}2^{20(p_0/\beta)m}),
\end{split}
\end{equation*}
in the support of the integral. After changes of variables, and recalling that $\mathcal{K}_\alpha=\frac{2\pi}{\gamma\alpha(\alpha-1)}$, see \eqref{nf50}, for \eqref{poi56} it suffices to prove that
\begin{equation}\label{poi57}
\Big|\int_{\R\times\R}e^{-iy_1y_2}\varphi_{\leq 0}(y_1/N)\varphi_{\leq 0}(y_2/N)\,dy_1 dy_2-2\pi\Big|\lesssim 2^{-m/2},
\end{equation}
where $N:=2^{\overline{l}}\sqrt{s\gamma\alpha(\alpha-1)/\xi^{2-\alpha}}$.

To prove \eqref{poi57} we start from the general identity
\begin{equation*}
\int_{\mathbb{R}}e^{-ax^2-bx}\,dx=e^{b^2/(4a)}\sqrt{\pi}/\sqrt{a},\qquad a,b\in\mathbb{C},\,\Re a>0.
\end{equation*}
Then we estimate, for $B\gg 1$,
\begin{equation*}
\int_{\mathbb{R}\times\mathbb{R}}e^{-ixy}e^{-x^2/B^2}e^{-y^2/B^2}\,dxdy=\sqrt{\pi}B\int_{\mathbb{R}}e^{-y^2/B^2}e^{-y^2B^2/4}\,dy=2\pi+O(B^{-1}).
\end{equation*}
Set $B=2^{2m}$. Since $N\in[2^{m/21},2^{m/19}]$, using integration by parts either in $y$ or in $x$, we have
\begin{equation*}
\begin{split}
&\Big|\int_{\mathbb{R}\times\mathbb{R}}e^{-ixy}e^{-x^2/B^2}e^{-y^2/B^2}\varphi_{\geq 1}(x/N)\varphi_{\leq 0}(y/N)\,dxdy\Big|\lesssim 2^{-m/2},\\
&\Big|\int_{\mathbb{R}\times\mathbb{R}}e^{-ixy}e^{-x^2/B^2}e^{-y^2/B^2}\varphi_{\geq 1}(y/N)\,dxdy\Big|\lesssim 2^{-m/2}.
\end{split}
\end{equation*}
Therefore 
\begin{equation*}
\Big|\int_{\mathbb{R}\times\mathbb{R}}e^{-ixy}e^{-x^2/B^2}e^{-y^2/B^2}\varphi_{\leq 0}(x/N)\varphi_{\leq 0}(y/N)\,dxdy-2\pi\Big|\lesssim 2^{-m/2},
\end{equation*}
and the desired conclusion \eqref{poi57} follows. This completes the proof of \eqref{poi53}.
\medskip

{\it{Proof of \eqref{poi54}.}} Without loss of generality we may assume that $l_1\geq \max(l_2,\overline{l}+1)$. We make the change of variables $\eta_1=\xi+x_1$, $\eta_2=\xi+x_2$, thus
\begin{equation}\label{poi60}
\begin{split}
J_{l_1,l_2}^{+,+,-}(\xi):=\int_{\R\times\R}m_1(\xi+x_1,&\xi+x_2,-\xi-x_1-x_2)e^{is\Phi(\xi,\xi+x_1,\xi+x_2)}\varphi_{l_1}^{(\overline{l})}(x_1)\varphi_{l_2}^{(\overline{l})}(x_2)\\
&\times\widehat{v_{k_1}^{+}}(\xi+x_1)\widehat{v_{k_2}^{+}}(\xi+x_2)\widehat{v_{k_3}^{-}}(-\xi-x_1-x_2)\,dx_1 dx_2.
\end{split}
\end{equation} 
We would like to integrate by parts in $x_2$. For this we notice that
\begin{equation}\label{poi61}
\begin{split}
\Big|\frac{d}{dx_2}[\Phi(\xi,\xi+x_1,\xi+x_2)]\Big|&=\big|\Lambda'(\xi+x_2)-\Lambda'(\xi+x_1+x_2)\big|\gtrsim 2^{l_1}2^{-k_3(2-\alpha)},
\end{split}
\end{equation}
in the support of the integral. We integrate by parts in $x_2$, as in the proof of Lemma \ref{lemma2}, and estimate
\begin{align*}
|J_{l_1,l_2}^{+,+,-}(\xi)| & \lesssim  |L_1(\xi)| + |L_2(\xi)| + |L_3(\xi)|,
\end{align*}
where, with $\Phi_\xi(x_1,x_2):=\Phi(\xi,\xi+x_1,\xi+x_2)=-\Lambda(\xi)+\Lambda(\xi+x_1)+\Lambda(\xi+x_2)+\Lambda(-\xi-x_1-x_2)$, 
\begin{align*}
\begin{split}
L_1(\xi) & := \int_{\mathbb{R}\times\mathbb{R}} e^{is\Phi_\xi(x_1,x_2)}
   n_2(\xi,x_1,x_2)
  \what{v_{k_1}^+}(\xi+x_1) (\partial\what{v_{k_2}^+})(\xi+x_2) \what{v_{k_3}^-}(-\xi-x_1-x_2) \,dx_1 dx_2,
\\
L_2(\xi) & := \int_{\mathbb{R}\times\mathbb{R}} e^{is\Phi_\xi(x_1,x_2)}
   n_2(\xi,x_1,x_2)
  \what{v_{k_1}^+}(\xi+x_1) \what{v_{k_2}^+}(\xi+x_2) (\partial\what{v_{k_3}^-})(-\xi-x_1-x_2) \,dx_1 dx_2,
\\
L_3(\xi) & := \int_{\mathbb{R}\times\mathbb{R}} e^{is\Phi_\xi(x_1,x_2)}
   (\partial_{x_2}n_2)(\xi,x_1,x_2)
  \what{v_{k_1}^+}(\xi+x_1) \what{v_{k_2}^+}(\xi+x_2) \what{v_{k_3}^-}(-\xi-x_1-x_2) \,dx_1 dx_2,
\end{split}
\end{align*}
and
\begin{align}\label{poi65}
\begin{split}
&n_2(\xi,x_1,x_2) := \frac{m_1(\xi+x_1,\xi+x_2,-\xi-x_1-x_2)}{s(\Lambda'(\xi+x_2)-\Lambda'(\xi+x_1+x_2))}\varphi_{l_1}^{(\overline{l})}(x_1)\varphi_{l_2}^{(\overline{l})}(x_2)\cdot\Psi_{k_1,k_2,k_3}(\xi,x_1,x_2),\\
&\Psi_{k_1,k_2,k_3}(\xi,x_1,x_2):=(\varphi'_{k_1}\cdot\mathbf{1}_+)(\xi+x_1)\cdot(\varphi'_{k_2}\cdot\mathbf{1}_+)(\xi+x_2)\cdot(\varphi'_{k_3}\cdot\mathbf{1}_+)(\xi+x_1+x_2).
\end{split}
\end{align}

We keep $\xi$ fixed and would like to use \eqref{mk6.1}. In view of \eqref{nonl7}, if $s\approx 2^m$ then
\begin{equation}\label{poi68.5}
\big\|\mathcal{F}^{-1}_{x_1,x_2}\{n_2(\xi,.,.)\}\big\|_{L^1(\mathbb{R}^2)}\lesssim 2^{-m}2^{-l_1}2^{6k_3^+}.
\end{equation}
Moreover, using \eqref{boundsflow}--\eqref{boundsfhigh},
\begin{equation}\label{poi69}
\begin{split}
&\|(\partial\what{v_{k_2}^+})(\xi+.)\|_{L^2}\lesssim \e_12^{p_0m}2^{-k_3},\\
&\|\what{v_{k_1}^+}(\xi+.)\varphi_{\leq l_1+4}(\xi+.)\|_{L^2}\lesssim \e_12^{l_1/2}2^{-N_2k_3^+}2^{10(p_0/\beta)m},\\
&\|\mathcal{F}^{-1}\{e^{is\Lambda(-\xi+.)}\what{v_{k_3}^-}(-\xi+.)\}\|_{L^\infty}\lesssim \e_12^{-m/2}2^{-k_3}.
\end{split}
\end{equation} 
Therefore, using \eqref{mk6.1},
\begin{equation*}
|L_1(\xi)|\lesssim \e_1^32^{-l_1/2}2^{-3m/2}2^{40(p_0/\beta)m}\lesssim \e_1^32^{-11m/10}.
\end{equation*}
A similar argument gives $|L_2(\xi)|\lesssim \e_1^32^{-11m/10}$. To bound $|L_3(\xi)|$ we also use \eqref{mk6.1}, and replace the $L^2$ bounds in \eqref{poi69} by
\begin{equation*}
\begin{split}
&\|\what{v_{k_1}^+}(\xi+.)\varphi_{\leq l_1+4}(\xi+.)\|_{L^2}\lesssim \e_12^{l_1/2}2^{-N_2k_3^+}2^{10(p_0/\beta)m},\\
&\|\what{v_{k_2}^+}(\xi+.)\varphi_{\leq l_2+4}(\xi+.)\|_{L^2}\lesssim \e_12^{l_2/2}2^{-N_2k_3^+}2^{10(p_0/\beta)m},
\end{split}
\end{equation*} 
and the $S^\infty$ bound in \eqref{poi68.5} by
\begin{equation*}
\big\|\mathcal{F}^{-1}_{x_1,x_2}\{\partial_{x_2}n_2(\xi,.,.)\}\big\|_{L^1(\mathbb{R}^2)}\lesssim 2^{-m}2^{-l_1}2^{-l_2}2^{6k_3^+}.
\end{equation*}
It follows that 
\begin{equation*}
|L_3(\xi)|\lesssim \e_1^32^{-l_1/2}2^{-l_2/2}2^{-3m/2}2^{40(p_0/\beta)m}\lesssim \e_1^32^{-51m/50}.
\end{equation*}
This completes the proof of \eqref{poi54}.
\end{proof}

\subsubsection{Proof of \eqref{mainlemma2conc3}} We estimate now the contribution of the quintic and higher order nonlinearities.

\begin{lemma}\label{rbounds}
For any $t\in[0,T]$ we have
\begin{equation}\label{nonl15}
\|\mathcal{R}_{\geq 2}(t)\|_{L^2}+\|S\mathcal{R}_{\geq 2}(t)\|_{L^2}\lesssim \varep_0(1+t)^{-3/2}.
\end{equation}
\end{lemma}

\begin{proof} It suffices to prove that, for any $n\geq 2$,
\begin{equation}\label{nonl16}
\|\mathcal{N}_{n}(t)\|_{L^2}+\|S\mathcal{N}_{n}(t)\|_{L^2}\leq (C\varep_1)^{2n+1}(1+t)^{-3/2}
\end{equation}
for some constant $C\geq 1$. We use the formula \eqref{Nonl3}. We notice that if $|y|\approx 2^p$ then
\begin{equation}\label{nonl17}
\Big\|\frac{H(x)-H(x-y)}{|y|}\Big\|_{L^q_x}\lesssim \min(2^{-p}\|H\|_{L^q},\|H'\|_{L^q})
\end{equation}
for any function $H:\mathbb{R}\to\mathbb{C}$ and $q\in \{2,\infty\}$. The bootstrap assumptions \eqref{zxc4} and the pointwise bounds \eqref{Poi2} show that
\begin{equation}\label{nonl18}
\begin{split}
&\|h(t)\|_{L^2}+\|h_x(t)\|_{L^2}+\|h_{xx}(t)\|_{L^2}+\|Sh(t)\|_{L^2}+\|Sh_x(t)\|_{L^2}+\|Sh_{xx}(t)\|_{L^2}\lesssim \varep_1\langle t\rangle^{p_0},\\
&\|h_x(t)\|_{L^\infty}+\|h_{xx}(t)\|_{L^\infty}\lesssim \varep_1\langle t\rangle^{-1/2},\qquad \|h(t)\|_{L^\infty}\lesssim \varep_1\langle t\rangle^{-1/5},
\end{split}
\end{equation}
for any $t\in[0,T]$ (the last bound follows from the estimates $\|P_kh\|_{L^\infty}\lesssim \varep_1\langle t\rangle^{-1/2}2^{\beta k-k/2}$ and $\|P_kh\|_{L^\infty}\lesssim \varep_12^{k/2}\langle t\rangle^{p_0}$). The bounds \eqref{nonl16} follow from \eqref{nonl17}--\eqref{nonl18} and the definitions. Indeed, let $\mathcal{N}_{n}^{(p)}$ denote the contribution of $|y|\approx 2^p$ in \eqref{Nonl3}. For $p\leq 0$ we use the estimate \eqref{nonl17} with $\|H'\|_{L^q}$, and estimate one of the factors in $L^2$ and the others in $L^\infty$ (for the vector-field bound we estimate the factor that carries the vector-field in $L^2$). It follows that
\begin{equation*}
\|\mathcal{N}^{(p)}_{n}(t)\|_{L^2}+\|S\mathcal{N}^{(p)}_{n}(t)\|_{L^2}\leq 2^{p(2-\alpha)}(C\varep_1)^{2n+1}(1+t)^{-3/2},
\end{equation*}
which gives gives the desired bound for the contribution of $p\leq 0$. The contribution over $p\geq 0$ can be bounded in a similar way: we estimate one of the factors using the $2^{-p}\|H\|_{L^q}$ bound in \eqref{nonl17}, and the remaining $2n$ factors using the $\|H'\|_{L^q}$ bound. The bound \eqref{nonl16} follows. 
\end{proof}

We turn now to the proof of \eqref{mainlemma2conc3}. Assume that $k\in[-10(p_0/\beta)m,10p_0m-1]$, $|\xi_0|\in [2^k,2^{k+1}]$, $m\geq 10$, $t_1\leq t_2\in[2^{m},2^{m+1}]\cap[0,T']$. We would like to prove that
\begin{equation}\label{lmj20}
\Big|\varphi_k(\xi_0)\int_{t_1}^{t_2}e^{iL(\xi_0,s)}e^{-is\Lambda(\xi_0)}\widehat{\mathcal{R}_{\geq 2}}(\xi_0,s)\,ds\Big|\lesssim \e_02^{-400p_0m}.
\end{equation}

Let
\begin{equation}\label{lmj21}
F(\xi):=\varphi_k(\xi)\int_{t_1}^{t_2}e^{iL(\xi_0,s)}e^{-is\Lambda(\xi)}\widehat{\mathcal{R}_{\geq 2}}(\xi,s)\,ds.
\end{equation}
In view of Lemma \ref{interpolation}, it suffices to prove that 
\begin{equation*}
2^{-k}\|F\|_{L^2}\big[2^k\|\partial F\|_{L^2}+\|F\|_{L^2}\big]\lesssim \e_0^22^{-800p_0m}.
\end{equation*}
Since $\|F\|_{L^2}\lesssim \e_02^{-m/2}$, see the first inequality in \eqref{nonl15}, it suffices to prove that
\begin{equation}\label{lmj22}
2^k\|\partial F\|_{L^2}\lesssim \e_02^{m/4}.
\end{equation}

To prove \eqref{lmj22} we write
\begin{equation*}
|\xi\partial_\xi F(\xi)|\leq |F_1(\xi)|+|F_2(\xi)|+|F_3(\xi)|,
\end{equation*}
where
\begin{equation*}
\begin{split}
&F_1(\xi):=\xi(\partial_\xi\varphi_k)(\xi)\int_{t_1}^{t_2}e^{iL(\xi_0,s)}\big[e^{-is\Lambda(\xi)}\widehat{\mathcal{R}_{\geq 2}}(\xi,s)\big]\,ds,\\
&F_2(\xi):=\varphi_k(\xi)\int_{t_1}^{t_2}e^{iL(\xi_0,s)}\big[\xi\partial_\xi-\alpha s\partial_s\big]\big[e^{-is\Lambda(\xi)}\widehat{\mathcal{R}_{\geq 2}}(\xi,s)\big]\,ds,\\
&F_3(\xi):=\varphi_k(\xi)\int_{t_1}^{t_2}e^{iL(\xi_0,s)}\alpha s\partial_s\big[e^{-is\Lambda(\xi)}\widehat{\mathcal{R}_{\geq 2}}(\xi,s)\big]\,ds.
\end{split}
\end{equation*}
Using \eqref{nonl15} again we have
\begin{equation*}
\|F_1\|_{L^2}+\|F_2\|_{L^2}\lesssim \e_0.
\end{equation*}
Moreover, using integration by parts in $s$ and the bound $\big|\partial_s\big[e^{iL(\xi_0,s)}\big]\big|\lesssim 2^{-9m/10}$, see the definition \eqref{nf50}, we can also estimate $\|F_3\|_{L^2}\lesssim \e_0$. The desired bound \eqref{lmj22} follows.

\bibliographystyle{abbrv}
\bibliography{references}

\def\cprime{$'$}
\begin{thebibliography}{10}

\bibitem{Alazard-Burq-Zuily:water-wave-surface-tension}
T.~Alazard, N.~Burq, and C.~Zuily.
\newblock On the water-wave equations with surface tension.
\newblock {\em Duke Math. J.}, 158(3):413--499, 2011.

\bibitem{ABZ2}
T.~Alazard, N.~Burq, and C.~Zuily.
\newblock On the {C}auchy problem for gravity water waves.
\newblock {\em Invent. Math.}, 198:71--163, 2014.

\bibitem{AD1}
T.~Alazard and J.~M. Delort.
\newblock Sobolev estimates for two dimensional gravity water waves.
\newblock {\em Ast\'{e}risque}, 374:viii+241 pages, 2015.

\bibitem{Buckmaster-Shkoller-Vicol:nonuniqueness-sqg}
T.~Buckmaster, S.~Shkoller, and V.~Vicol.
\newblock Nonuniqueness of weak solutions to the {S}{Q}{G} equation.
\newblock {\em Arxiv preprint arXiv:1610.00676}, 2016.

\bibitem{Castro-Cordoba-GomezSerrano:existence-regularity-vstates-gsqg}
A.~Castro, D.~C{\'o}rdoba, and J.~G{\'o}mez-Serrano.
\newblock Existence and regularity of rotating global solutions for the
  generalized surface quasi-geostrophic equations.
\newblock {\em Duke Math. J.}, 165(5):935--984, 2016.

\bibitem{Castro-Cordoba-GomezSerrano:global-smooth-solutions-sqg}
A.~Castro, D.~C{\'o}rdoba, and J.~G{\'o}mez-Serrano.
\newblock Global smooth solutions for the inviscid {S}{Q}{G} equation.
\newblock {\em Arxiv preprint arXiv:1603.03325}, 2016.

\bibitem{Castro-Cordoba-GomezSerrano:analytic-vstates-ellipses}
A.~Castro, D.~C{\'o}rdoba, and J.~G{\'o}mez-Serrano.
\newblock Uniformly rotating analytic global patch solutions for active
  scalars.
\newblock {\em Annals of PDE}, 2(1):1--34, 2016.

\bibitem{Castro-Cordoba-GomezSerrano:uniformly-rotating-smooth-euler}
A.~Castro, D.~C{\'o}rdoba, and J.~G{\'o}mez-Serrano.
\newblock Uniformly rotating smooth solutions for the incompressible 2{D}
  {E}uler equations.
\newblock {\em Arxiv preprint arXiv:1612.08964}, 2016.

\bibitem{Chae-Constantin-Cordoba-Gancedo-Wu:gsqg-singular-velocities}
D.~Chae, P.~Constantin, D.~C{\'o}rdoba, F.~Gancedo, and J.~Wu.
\newblock Generalized surface quasi-geostrophic equations with singular
  velocities.
\newblock {\em Comm. Pure Appl. Math.}, 65(8):1037--1066, 2012.

\bibitem{CK}
D.~Christodoulou and S.~Klainerman.
\newblock {\em The global nonlinear stability of the Minkowski space}.
\newblock Princeton Mathematical Series 41. Princeton University Press,
  Princeton, NJ, 1993.

\bibitem{Constantin-Lai-Sharma-Tseng-Wu:new-numerics-sqg}
P.~Constantin, M.-C. Lai, R.~Sharma, Y.-H. Tseng, and J.~Wu.
\newblock New numerical results for the surface quasi-geostrophic equation.
\newblock {\em J. Sci. Comput.}, 50(1):1--28, 2012.

\bibitem{Constantin-Majda-Tabak:formation-fronts-qg}
P.~Constantin, A.~J. Majda, and E.~Tabak.
\newblock Formation of strong fronts in the {$2$}-{D} quasigeostrophic thermal
  active scalar.
\newblock {\em Nonlinearity}, 7(6):1495--1533, 1994.

\bibitem{Cordoba:nonexistence-hyperbolic-blowup-qg}
D.~Cordoba.
\newblock Nonexistence of simple hyperbolic blow-up for the quasi-geostrophic
  equation.
\newblock {\em Ann. of Math. (2)}, 148(3):1135--1152, 1998.

\bibitem{Cordoba-Fefferman:growth-solutions-qg-2d-euler}
D.~Cordoba and C.~Fefferman.
\newblock Growth of solutions for {QG} and 2{D} {E}uler equations.
\newblock {\em J. Amer. Math. Soc.}, 15(3):665--670, 2002.

\bibitem{Cordoba-Fontelos-Mancho-Rodrigo:evidence-singularities-contour-dynamics}
D.~C{\'o}rdoba, M.~A. Fontelos, A.~M. Mancho, and J.~L. Rodrigo.
\newblock Evidence of singularities for a family of contour dynamics equations.
\newblock {\em Proc. Natl. Acad. Sci. USA}, 102(17):5949--5952, 2005.

\bibitem{Deem-Zabusky:vortex-waves-stationary}
G.~S. Deem and N.~J. Zabusky.
\newblock Vortex waves: Stationary "{V}-states", interactions, recurrence, and
  breaking.
\newblock {\em Physical Review Letters}, 40(13):859--862, 1978.

\bibitem{DelortKGE}
J.~M. Delort.
\newblock Global existence and asymptotic behavior for the quasilinear
  {K}lein-{G}ordon equation with small data in dimension 1.
\newblock {\em Ann. Sci. \'{E}cole Norm. Sup.}, 34:1--61, 2001.

\bibitem{Deng-Hou-Li-Yu:non-blowup-2d-sqg}
J.~Deng, T.~Y. Hou, R.~Li, and X.~Yu.
\newblock Level set dynamics and the non-blowup of the 2{D} quasi-geostrophic
  equation.
\newblock {\em Methods Appl. Anal.}, 13(2):157--180, 2006.

\bibitem{DIP}
Y.~Deng, A.~D. Ionescu, and B.~Pausader.
\newblock The {E}uler-{M}axwell system for electrons: global solutions in 2d.
\newblock {\em ArXiv preprint arXiv:1605.05340}, 2015.

\bibitem{Deng-Ionescu-Pausader-Pusateri:global-solutions-gravity-capillary-water-waves-3d}
Y.~Deng, A.~D. Ionescu, B.~Pausader, and F.~Pusateri.
\newblock Global solutions of the gravity-capillary water wave system in 3
  dimensions.
\newblock {\em ArXiv preprint arXiv:1601.05685}, 2016.

\bibitem{Dritschel:exact-rotating-solution-sqg}
D.~G. Dritschel.
\newblock An exact steadily rotating surface quasi-geostrophic elliptical
  vortex.
\newblock {\em Geophys. Astrophys. Fluid Dyn.}, 105(4-5):368--376, 2011.

\bibitem{Elcrat-Fornberg-Miller:stability-vortices-cylinder}
A.~Elcrat, B.~Fornberg, and K.~Miller.
\newblock Stability of vortices in equilibrium with a cylinder.
\newblock {\em Journal of Fluid Mechanics}, 544:53--68, 2005.

\bibitem{Gancedo:existence-alpha-patch-sobolev}
F.~Gancedo.
\newblock Existence for the {$\alpha$}-patch model and the {QG} sharp front in
  {S}obolev spaces.
\newblock {\em Adv. Math.}, 217(6):2569--2598, 2008.

\bibitem{Gancedo-Strain:absence-splash-muskat-SQG}
F.~{Gancedo} and R.~M. {Strain}.
\newblock Absence of splash singularities for surface quasi-geostrophic sharp
  fronts and the muskat problem.
\newblock {\em Proceedings of the National Academy of Sciences},
  111(2):635--639, 2014.

\bibitem{GeMaSh}
P.~Germain, N.~Masmoudi, and J.~Shatah.
\newblock Global solutions for 3d quadratic {S}chr\"{o}dinger equations.
\newblock {\em Int. Math. Res. Not.}, 2009:414--432, 2009.

\bibitem{Germain-Masmoudi-Shatah:global-solutions-gravity-water-waves-annals}
P.~Germain, N.~Masmoudi, and J.~Shatah.
\newblock Global solutions for the gravity water waves equation in dimension 3.
\newblock {\em Ann. of Math. (2)}, 175:691--754, 2012.

\bibitem{GIP}
Y.~Guo, A.~D. Ionescu, and B.~Pausader.
\newblock Global solutions of the {E}uler-{M}axwell two-fluid system in 3d.
\newblock {\em Ann. of Math. (2)}, 183:377--498, 2016.

\bibitem{GNT1}
S.~Gustafson, K.~Nakanishi, and T.~Tsai.
\newblock Scattering for the {G}ross-{P}itaevsky equation in 3 dimensions.
\newblock {\em Commun. Contemp. Math.}, 11:657--707, 2009.

\bibitem{Hassainia-Hmidi:v-states-generalized-sqg}
Z.~Hassainia and T.~Hmidi.
\newblock On the {V}-states for the generalized quasi-geostrophic equations.
\newblock {\em Comm. Math. Phys.}, 337(1):321--377, 2015.

\bibitem{Held-Pierrehumbert-Garner-Swanson:sqg-dynamics}
I.~M. Held, R.~T. Pierrehumbert, S.~T. Garner, and K.~L. Swanson.
\newblock Surface quasi-geostrophic dynamics.
\newblock {\em J. Fluid Mech.}, 282:1--20, 1995.

\bibitem{IP1}
A.~D. Ionescu and B.~Pausader.
\newblock The {E}uler-{P}oisson system in 2d: global stability of the constant
  equilibrium solution.
\newblock {\em Int. Math. Res. Not.}, 2013:761--826, 2013.

\bibitem{Ionescu-Pusateri:global-solutions-water-waves-2d-surface-tension}
A.~D. Ionescu and F.~Pusateri.
\newblock Global regularity for 2d water waves with surface tension.
\newblock {\em ArXiv preprint arXiv:1408.4428}, 2014.

\bibitem{Ionescu-Pusateri:global-solutions-water-waves-2d}
A.~D. Ionescu and F.~Pusateri.
\newblock Global solutions for the gravity water waves system in 2{D}.
\newblock {\em Invent. Math.}, 199:653--804, 2015.

\bibitem{Ionescu-Pusateri:model2d}
A.~D. Ionescu and F.~Pusateri.
\newblock Global analysis of a model for capillary water waves in two
  dimensions.
\newblock {\em Comm. Pure Appl. Math.}, 69:2015--2071, 2016.

\bibitem{Kiselev-Nazarov:simple-energy-pump-sqg}
A.~Kiselev and F.~Nazarov.
\newblock A simple energy pump for the surface quasi-geostrophic equation.
\newblock In H.~Holden and K.~H. Karlsen, editors, {\em Nonlinear Partial
  Differential Equations}, volume~7 of {\em Abel Symposia}, pages 175--179.
  Springer Berlin Heidelberg, 2012.

\bibitem{Kiselev-Ryzhik-Yao-Zlatos:singularity-alpha-patch-boundary}
A.~Kiselev, L.~Ryzhik, Y.~Yao, and A.~Zlato\v{s}.
\newblock Finite time singularity formation for the modified {S}{Q}{G} patch
  equation.
\newblock {\em Ann. of Math. (2)}, 184:909--948, 2016.

\bibitem{Klainerman:vector-fields-wave-equation}
S.~Klainerman.
\newblock Uniform decay estimates and the {L}orentz invariance of the classical
  wave equation.
\newblock {\em Comm. Pure Appl. Math.}, 38(3):321--332, 1985.

\bibitem{Klainerman:null-structure-global-existence-wave-equation}
S.~Klainerman.
\newblock The null condition and global existence to nonlinear wave equations.
\newblock In {\em Nonlinear systems of partial differential equations in
  applied mathematics, {P}art 1 ({S}anta {F}e, {N}.{M}., 1984)}, volume~23 of
  {\em Lectures in Appl. Math.}, pages 293--326. Amer. Math. Soc., Providence,
  RI, 1986.

\bibitem{LuzzattoFegiz-Williamson:efficient-numerical-method-steady-uniform-vortices}
P.~Luzzatto-Fegiz and C.~H.~K. Williamson.
\newblock An efficient and general numerical method to compute steady uniform
  vortices.
\newblock {\em Journal of Computational Physics}, {230}({17}):{6495--6511},
  2011.

\bibitem{Marchand:existence-regularity-weak-solutions-sqg}
F.~Marchand.
\newblock Existence and regularity of weak solutions to the quasi-geostrophic
  equations in the spaces {$L^p$} or {$\dot H^{-1/2}$}.
\newblock {\em Comm. Math. Phys.}, 277(1):45--67, 2008.

\bibitem{Nahmod-Pavlovic-Staffilani-Totz:global-invariant-measures-gsqg}
A.~Nahmod, N.~Pavlovic, G.~Staffilani, and N.~Totz.
\newblock Global flows with invariant measures for the inviscid modified
  {S}{Q}{G} equations.
\newblock {\em ArXiv preprint arXiv:1705.01890}, 2017.

\bibitem{Resnick:phd-thesis-sqg-chicago}
S.~G. Resnick.
\newblock {\em Dynamical problems in non-linear advective partial differential
  equations}.
\newblock PhD thesis, University of Chicago, Department of Mathematics, 1995.

\bibitem{Rodrigo:evolution-sharp-fronts-qg}
J.~L. Rodrigo.
\newblock On the evolution of sharp fronts for the quasi-geostrophic equation.
\newblock {\em Comm. Pure Appl. Math.}, 58(6):821--866, 2005.

\bibitem{Saffman-Szeto:equilibrium-shapes-equal-uniform-vortices}
P.~Saffman and R.~Szeto.
\newblock {Equilibrium shapes of a pair of equal uniform vortices}.
\newblock {\em {Physics of Fluids}}, {23}({12}):{2339--2342}, {1980}.

\bibitem{Scott:scenario-singularity-quasigeostrophic}
R.~K. Scott.
\newblock A scenario for finite-time singularity in the quasigeostrophic model.
\newblock {\em Journal of Fluid Mechanics}, 687:492--502, 11 2011.

\bibitem{Scott-Dritschel:self-similar-sqg}
R.~K. Scott and D.~G. Dritschel.
\newblock Numerical simulation of a self-similar cascade of filament
  instabilities in the surface quasigeostrophic system.
\newblock {\em Phys. Rev. Lett.}, 112:144505, 2014.

\bibitem{Shatah:normal-forms-quadratic-klein-gordon}
J.~Shatah.
\newblock Normal forms and quadratic nonlinear {K}lein-{G}ordon equations.
\newblock {\em Comm. Pure Appl. Math.}, 38(5):685--696, 1985.

\bibitem{Si}
J.~Simon.
\newblock A wave operator for a nonlinear {K}lein--{G}ordon equation.
\newblock {\em Lett. Math. Phys.}, 7:387--398, 1983.

\bibitem{Wu-Overman-Zabusky:steady-state-Euler-2d}
H.~M. Wu, E.~A. Overman, II, and N.~J. Zabusky.
\newblock Steady-state solutions of the {E}uler equations in two dimensions:
  rotating and translating {$V$}-states with limiting cases. {I}. {N}umerical
  algorithms and results.
\newblock {\em J. Comput. Phys.}, 53(1):42--71, 1984.

\end{thebibliography}

\end{document}